\documentclass[reqno]{amsart}

\usepackage{amsmath}
\usepackage{amsthm}
\usepackage{amssymb}
\usepackage{mathrsfs} %script font
\usepackage{bm,color}
\usepackage{courier}
\usepackage[shortlabels]{enumitem}
    \setlist[enumerate,1]{nosep,font=\normalfont}
\usepackage{fnpct}
\usepackage{stmaryrd}
\usepackage[dvipsnames]{xcolor}
\usepackage{subcaption}
\usepackage{stmaryrd}
\usepackage{hyperref}
    \hypersetup{
        colorlinks,
        linkcolor={blue!50!white},
        citecolor={blue!50!white},
        urlcolor={blue!80!white}
    }
\usepackage{comment}
\usepackage{tikz,pgfplots}
\pgfplotsset{compat=newest,compat/show suggested version=false}
\usetikzlibrary{arrows.meta}
\usetikzlibrary{shapes.misc}
\usepackage{ifthen}
\usepackage{tkz-graph}

\pgfdeclarelayer{background}
\pgfdeclarelayer{foreground}
\pgfsetlayers{background,main,foreground}

\numberwithin{equation}{section}

\newcommand{\R}{\mathbb{R}}
\newcommand{\Q}{\mathbb{Q}}
\newcommand{\Z}{\mathbb{Z}}

\renewcommand{\P}{\mathcal{P}}

\newcommand{\diam}{\mathrm{diam}}

\newcommand{\supp}{\mathrm{supp}}
\newcommand{\Rh}{\mathrm{Rh}}
\newcommand{\Up}{\mathrm{Up}}

\newcommand{\pert}{\mathrm{pert}}
\newcommand{\descent}{\textsc{Descending Local Search}}
\newcommand{\ascent}{\textsc{Ascending Local Search}}

\newcommand{\supermaj}{\prec^{\mathrm{w}}}
\newcommand{\submaj}{\prec_{\mathrm{w}}}
\newcommand{\maj}{\prec}

\newcommand{\floor}[1]{\left\lfloor #1\right\rfloor}
\newcommand{\ceil}[1]{\left\lceil #1\right\rceil}

\newcommand{\spn}{\textrm{span}}
\renewcommand{\d}{d^*}

\renewcommand{\emptyset}{\varnothing}
\newcommand{\st}{:}
\renewcommand{\mod}{\bmod}

\newcommand{\spec}{\sigma^*}

\newcommand{\cw}[3]{#1^{cw}\left\{#2,\ldots,#3\right\}}
\newcommand{\mult}[2]{\mu\left(#1,#2\right)}

\definecolor{mycolor}{RGB}{252,150,255}

\newtheorem{theorem}{Theorem}[section]
\newtheorem{lemma}[theorem]{Lemma}
\newtheorem{corollary}[theorem]{Corollary}
\newtheorem{proposition}[theorem]{Proposition}

\theoremstyle{definition}

\newtheorem{question}[theorem]{Question}
\newtheorem{remark}[theorem]{Remark}

\newtheorem{definition}[theorem]{Definition}
\newtheorem{example}[theorem]{Example}

\definecolor{gren}{RGB}{20, 140, 20}

\thanks{All three authors were partially supported by a Virginia Commonwealth University Seed Award during 2023 and 2024.}
\date{\today}

\begin{document}

\title{Sets of vertices with extremal energy}

\author[Neal Bushaw]{Neal Bushaw}
\address[Neal Bushaw]{ 
Virginia Commonwealth University,
Department of Mathematics and Applied Mathematics,
1015 Floyd Avenue, PO Box 842014, Richmond, Virginia 23284, United States
} 
\email[N. ~Bushaw]{nobushaw@vcu.edu} 
\urladdr{http://thenealon.github.io/}

\author[Brent Cody]{Brent Cody}
\address[Brent Cody]{ 
Virginia Commonwealth University,
Department of Mathematics and Applied Mathematics,
1015 Floyd Avenue, PO Box 842014, Richmond, Virginia 23284, United States
} 
\email[B. ~Cody]{bmcody@vcu.edu} 
\urladdr{http://www.people.vcu.edu/~bmcody/}

\author[Chris Leffler]{Chris Leffler}
\address[Chris Leffler]{ 
Virginia Commonwealth University,
Department of Mathematics and Applied Mathematics,
1015 Floyd Avenue, PO Box 842014, Richmond, Virginia 23284, United States
} 
\email[C. ~Leffler]{lefflerc@vcu.edu} 
%\urladdr{http://www.people.vcu.edu/~bmcody/}

\thanks{}

% TO DO LIST
% Typo: page 3 "sets of $X$"
% I think Corollary 3.7 isn't right. For example, some graphs might not have minimizers that work for all decreasing convex functions (like paths). Hence, the majorization orders will not have least elements, although they will have various minimal elements.

\begin{abstract}
We define various notions of energy of a set of vertices in a graph, which generalize two of the most widely studied graphical indices: the Wiener index and the Harary index. We provide a new proof of a result due to Douthett and Krantz, which says that for cycles, the sets of vertices which have minimal energy among all sets of the same size are precisely the \emph{maximally even sets}, as defined in Clough and Douthett's work on music theory. Generalizing a theorem of Clough and Douthett, we prove that a finite, simple, connected graph is distance degree regular if and only if whenever a set of vertices has minimal energy, its complement also has minimal energy. We also provide several characterizations of sets of vertices in finite paths and cycles for which the sum of all pairwise distances between vertices in the set is maximal among all sets of the same size.
\end{abstract}



\maketitle

%\tableofcontents

\section{Introduction}\label{section_intro}

If a finite number of unit point charges are constrained to move on the surface of a sphere, the electrostatic force will be repulsive, and in the presence of an appropriate viscous force, the system will evolve so that as the kinetic energy dissipates, the positions of the particles will approach those of a configuration of minimal potential energy. The task of identifying such minimal energy configurations on the $2$-sphere $S^2\subseteq\mathbb{R}^3$, known as Thomson's problem \cite{MR2555698,MR1439152,MR3047910}, has proven to be quite difficult and is still largely unfinished with exact solutions known in only a handful of cases. A variation of Thomson's problem appears 7th on Steve Smale's list \cite{MR1754762, MR1631413} of mathematical problems for the 21st century. Generalizations of Thomson's problem to higher dimensions, and to interactions other than the electrostatic force, have deep connections \cite{MR2257398} with the widely popular subject of sphere packing \cite{MR1662447}. In this article, we introduce a discrete version of this kind of energy minimization problem in which particles are constrained to the vertices of a finite connected graph and distance calculations are done in an associated finite metric space.

%We use standard graph theoretic notation throughout (see, e.g., \cite{MGT}). In particular, let $G$ be a finite simple connected graph with vertex set $V(G)$. For vertices $u$ and $v$ in $G$, the \emph{distance} from $u$ to $v$, denoted by $d(u,v)$, is the length of a shortest path from $u$ to $v$ in $G$. The \emph{diameter} of $G$, written $\diam(G)$, is the largest distance between any pair of vertices in $V(G)$.

Suppose $G$ is a finite simple connected graph. For vertices $u$ and $v$ in $G$, the \emph{distance} from $u$ to $v$, denoted by $d(u,v)$, is the length of a shortest path from $u$ to $v$ in $G$. The \emph{diameter} of $G$, written $\diam(G)$, is the largest distance between two vertices in $V(G)$ (we use standard graph theoretic terminology and notation throughout, see e.g., \cite{MGT}). Suppose $g:[1,\infty)\to \R$ is a function. For a set $A$ of vertices in $G$, we let the \emph{$g$-energy of $A$ in $G$} be the quantity 
\[E_g(A)=\sum_{\{u,v\}\in\binom{A}{2}}g\left(d\left(u,v\right)\right),\]
where $\binom{A}{2}=\{X\subseteq A\st |A|=2\}$ and
where $E_g(A)=0$ if $|A|\leq 1$. When the context is clear, we refer to $E_g(A)$ as the \emph{energy of $A$}. By choosing $g$ appropriately, we obtain several well-studied graphical indices as $E_g(V(G))$. When $g$ is the identity function on $\{1,\ldots,\diam(G)\}$ the quantity $E_g(V(G))$ is the \emph{Wiener index of $G$}, denoted by $W(G)$, and when $g$ is the function defined by $g(r)=\frac{1}{r}$ the quantity $E_g(V(G))$ is the \emph{Harary index of $G$}, denoted by $H(G)$. Thus, for a set of vertices $A$, we define the \emph{Wiener index of $A$} (in $G$) to be $W(A)=E_g(A)$ where $g(r)=r$, and similarly, the \emph{Harary index of $A$} (in $G$) is $H(A)=E_g(A)$ where $g(r)=\frac{1}{r}$. Despite the fact that both the Wiener index \cite{MR1843259,MR1916949,MR1461039,MR3570470} and the Harary index \cite{MR1219576,MR2823920,Xu_Das_Trinajstic_harary_index} of graphs have been widely studied in both mathematics and chemistry, there seems to be little to no literature on generalizations of these indices to subsets of the vertex set of a graph, or generalizations involving functions $g$ other than the identity function and $g(r)=\frac{1}{r}$.

Suppose $G$ is a finite simple connected graph and let $A$ and $B$ be sets of vertices in $G$ with $|A|=|B|$. Then $W(A)>W(B)$ indicates that the average distance between pairs of elements of $A$ is greater than the average distance between pairs of elements of $B$. Furthermore, when $g:[1,\infty)\to \R$ is a strictly decreasing strictly convex function, such as $g(r)=\frac{1}{r}$, the expression defining $E_g(A)$ resembles the electric potential energy of a system of unit point charges, and $E_g(A)<E_g(B)$ indicates that the vertices in $A$ are ``more spread out'' than those of $B$. Hence, intuitively, when $E_g(A)=\min\{E_g(B)\st B\subseteq V(G)\land |B|=|A|\}$, the vertices in $A$ should be ``spread out as much as possible'' in $G$ (see Figure \ref{figure_examples} and Figure \ref{figure_mobius}). As we will see below $E_g$, where $g$ is strictly decreasing and strictly convex, tends to be a more discerning measure of how spread out a set of vertices is, as opposed to $W$, since it is often the case that there will be many sets of the same cardinality with the same Wiener index, but which all have different $E_g$-values (see Theorem \ref{theorem_min_energy_cycles} and Corollary \ref{corollary_relationship}). For example, up to rotations and reflections there are three maximizers of $W$ on the $8$-cycle $C_8$, see Figure \ref{figure_WI_max_c8}, all of which have different values of $E_g$ where $g(r)=\frac{1}{r}.$

%Notice that for two sets of vertices $A$ and $B$ with $|A|=|B|$, when $W(A)>W(B)$ it follows that the average distances between pairs of elements of $A$ is greater than that of $B$, and so the vertices in $A$ are, in a sense, ``more spread out'' than those of $B$. A more discerning measure of how spread out a set of vertices is, can be obtained by taking $g$ to be a strictly decreasing strictly convex function, such as $g(r)=\frac{1}{r}$. In this case, when $E_g(A)<E_g(B)$, where $A$ and $B$ are two sets of vertices of the same cardinality, then the vertices in $A$ are, intuitively speaking, ``more spread out'' than those of $B$. Since the electric potential energy of a system of unit point charges is proportional to the sum of the reciprocals of distances between pairs of particles in the system, when $g(r)=\frac{1}{r}$, the expression defining $E_g(A)$ resembles the electric potential energy of a system of unit charges constrained to the points in the metric space $(V(G),d)$. Hence, intuitively, when $E_g(A)=\min\{E_g(B)\st B\subseteq V(G)\land |B|=|A|\}$, the vertices in $A$ should be ``spread out as evenly as possible'' in $G$ (see Figure \ref{figure_examples} and Figure \ref{figure_mobius}).

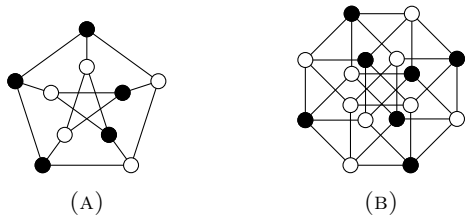
\begin{figure}
\centering
\begin{subfigure}{.3\textwidth}
    \centering
\begin{tikzpicture}[scale=0.5]

% Define vertices of the Petersen graph
\coordinate (A) at (18:1);
\coordinate (B) at (90:1);
\coordinate (C) at (162:1);
\coordinate (D) at (234:1);
\coordinate (E) at (306:1);
\coordinate (F) at (18:2);
\coordinate (G) at (90:2);
\coordinate (H) at (162:2);
\coordinate (I) at (234:2);
\coordinate (J) at (306:2);

% Draw edges of the Petersen graph
\foreach \x/\y in {F/G, G/H, H/I, I/J, J/F,  A/F, B/G, C/H, D/I, E/J,  A/C, B/D, C/E, D/A, E/B}
    \draw (\x) -- (\y);

% Draw vertices
\foreach \x in {A,B,C,D,E,F,G,H,I,J}
    \node[draw,circle,inner sep=2, fill=white] at (\x) {};

% Draw shaded ertices
\foreach \x in {A,E,G,H,I}
    \node[draw,circle,inner sep=2, fill=black] at (\x) {};
    
\end{tikzpicture}
    \caption{\small }
    \label{figure_peter}
\end{subfigure}
\begin{subfigure}{.3\textwidth}
    \centering
\begin{tikzpicture}[scale = 0.4]

\coordinate (1) at (0,0,0);
\coordinate (2) at (2,0,0);
\coordinate (3) at (2,2,0);
\coordinate (4) at (0,2,0);
\coordinate (5) at (0,0,4);
\coordinate (6) at (2,0,4);
\coordinate (7) at (2,2,4);
\coordinate (8) at (0,2,4);

\coordinate (a) at (0+1.5,0-1.5,0);
\coordinate (b) at (2+1.5,0-1.5,0);
\coordinate (c) at (2+1.5,2-1.5,0);
\coordinate (d) at (0+1.5,2-1.5,0);
\coordinate (e) at (0+1.5,0-1.5,4);
\coordinate (f) at (2+1.5,0-1.5,4);
\coordinate (g) at (2+1.5,2-1.5,4);
\coordinate (h) at (0+1.5,2-1.5,4);

\foreach \i/\j in {1/2,2/3,3/4,4/1, 5/6,6/7,7/8,8/5, 1/5,2/6,3/7,4/8}
    \draw (\i) -- (\j);

\foreach \i/\j in {a/b,b/c,c/d,d/a, e/f,f/g,g/h,h/e, a/e,b/f,c/g,d/h}
    \draw (\i) -- (\j);

\foreach \i/\j in {1/a,2/b,3/c,4/d,5/e,6/f,7/g,8/h}
    \draw (\i) -- (\j);

\foreach \i in {1,2,3,4,5,6,7,8}
    \node[draw,circle,inner sep=2, fill=white] at (\i) {};

\foreach \i in {a,b,c,d,e,f,g,h}
    \node[draw,circle,inner sep=2, fill=white] at (\i) {};

\foreach \i in {2,4,5,7}
    \node[draw,circle,inner sep=2, fill=black] at (\i) {};

\foreach \i in {a,c,f}
    \node[draw,circle,inner sep=2, fill=black] at (\i) {};

\end{tikzpicture}
    \caption{\small }
    \label{figure_hyper}
\end{subfigure}
\caption{\small \small Examples of minimizers of $E_g$, where $g(r)=\frac{1}{r}$, in the Petersen graph (A) and the hypercube (B). %The shaded set of vertices (A) is a minimizer of $E_g$ on the Petersen graph where $g(r)=\frac{1}{r}$, and similarly, that in iA minimizercardinality five on the Petersen graph (A) and a minimizer of on the hypercube (B)}\label{figure_locnotglo}
}\label{figure_examples}
\end{figure}

\begin{definition} Suppose $G$ is a finite simple connected graph and  $F:P(V(G))\to\R$ is a real-valued function defined on the powerset of the vertex set of $G$. We say that a set of vertices $A$ is a \emph{global minimizer of $F$} (or just a \emph{minimizer of $F$}) if 
\[F(A)=\min\{F(B)\st B\subseteq V(G)\text{ and } |B|=|A|\},\]
and similarly, $A$ is a \emph{global maximizer of $F$} (or just a \emph{maximizer of $F$}) if 
\[F(A)=\max\{F(B)\st B\subseteq V(G)\text{ and } |B|=|A|\}.\]
\end{definition}
Let us note that one can determine the minimizers of $F$ on $G$ using the obvious brute force algorithm, which takes as input $(G,F,m)$, computes $F(X)$ for all $X\in\binom{V(G)}{m}$ and outputs the sets of $X$ for which $F(X)$ is minimal. After completing the present article the authors were made aware of the work of Barrett \cite{Barrett}, who considered the problem of finding minimizers of $E_g$ on various graphs in the particular case in which $g(r)=\frac{1}{r}$. For example, Barrett showed \cite[Section 3.3]{Barrett} that in this setting, the problem of finding minimizers of $E_g$ is NP-complete. Many of the results in the current article provide faster algorithms for determining the minimizers of functions of the form $E_g$ on various classes of graphs and various functions $g$.

\begin{definition}
For two sets $A$ and $B$ of vertices in $G$, we say that $B$ is a \emph{perturbation of $A$ in $G$} if $B=(A\setminus\{u\})\cup \{v\}$ for some $u\in A$ and some $v\in V(G)\setminus A$ with $\{u,v\}\in E(G)$. Let $\pert_G(A)$ denote the set of all perturbations of $A$ in $G$. We say that a set of vertices $A$ in $G$ is a \emph{local minimizer of $F$} if 
\[F(A)=\min\{F(B)\st B\in \pert_G(A)\},\]
and $A$ is a \emph{local maximizer of $F$} if
\[F(A)=\max\{F(B)\st B\in \pert_G(A)\}.\]
\end{definition}
Clearly every minimizer (maximizer) of $F$ is a local minimizer (maximizer), but, as demonstrated by the following example, the converse does not hold in general. However, in what follows, we prove that for paths (Proposition \ref{proposition_wiener_path}) and cycles (Theorem \ref{theorem_w}), every local maximizer of $W$ is a global maximizer of $W$.
\begin{figure}
\centering
\begin{subfigure}{.22\textwidth}
    \centering
\begin{tikzpicture}
\tikzset{highlight/.style={preaction={%But before that
draw,mycolor,-,% Draw yellow without any arrow head
double=mycolor,
double distance=10\pgflinewidth,opacity=0.3
}}}

\def \rotation {360/7};
\def \diameter {1};

\foreach \i [count=\xi from 0, remember=\i as \iprev] in {0,...,4} {
    \node[draw,circle,inner sep=2] (\i) at ({\diameter*cos(360*\i/5-\rotation)},{(\diameter/1.5)*sin(360*\i/5-\rotation)}) {};
    \ifthenelse{\i>0}{
        \draw[-] (\iprev) -- (\i);}{}
        }

\foreach \i [count=\xi from 0, remember=\i as \iprev] in {0,...,4}{
  \node[draw,circle,inner sep=2, fill=white] (a\i) at ({\diameter*cos(360*\i/5-\rotation)},{(\diameter/1.5)*sin(360*\i/5-\rotation)+0.7}) {};

  \draw[-] (\i) -- (a\i);

  \ifthenelse{\i>0}{
    \draw[-] (a\iprev) -- (a\i);
  }{}

}

\draw[-] (0) -- (4);
\draw[-] (a0) -- (a4);

\foreach \i in {0,1,3}{
    \node[draw,circle,inner sep=2, fill=black] (b\i) at ({\diameter*cos(360*\i/5-\rotation)},{(\diameter/1.5)*sin(360*\i/5-\rotation)+0.7}) {};
}

\foreach \i in {2,4}{
    \node[draw,circle,inner sep=2,fill=black] (c\i) at ({\diameter*cos(360*\i/5-\rotation)},{(\diameter/1.5)*sin(360*\i/5-\rotation)}) {};
}

\end{tikzpicture}
    \caption{\small }
    \label{figure_locnotglo_b}
\end{subfigure}
\begin{subfigure}{.22\textwidth}
    \centering
\begin{tikzpicture}
\tikzset{highlight/.style={preaction={%But before that
draw,mycolor,-,% Draw yellow without any arrow head
double=mycolor,
double distance=10\pgflinewidth,opacity=0.3
}}}

\def \rotation {360/7};
\def \diameter {1};

\foreach \i [count=\xi from 0, remember=\i as \iprev] in {0,...,4} {
    \node[draw,circle,inner sep=2] (\i) at ({\diameter*cos(360*\i/5-\rotation)},{(\diameter/1.5)*sin(360*\i/5-\rotation)}) {};
    \ifthenelse{\i>0}{
        \draw[-] (\iprev) -- (\i);}{}
        }

\foreach \i [count=\xi from 0, remember=\i as \iprev] in {0,...,4}{
  \node[draw,circle,inner sep=2, fill=white] (a\i) at ({\diameter*cos(360*\i/5-\rotation)},{(\diameter/1.5)*sin(360*\i/5-\rotation)+0.7}) {};

  \draw[-] (\i) -- (a\i);

  \ifthenelse{\i>0}{
    \draw[-] (a\iprev) -- (a\i);
  }{}

}

\draw[-] (0) -- (4);
\draw[-] (a0) -- (a4);

\foreach \i in {0,1,3}{
    \node[draw,circle,inner sep=2, fill=black] (b\i) at ({\diameter*cos(360*\i/5-\rotation)},{(\diameter/1.5)*sin(360*\i/5-\rotation)+0.7}) {};
}

\foreach \i in {0,3}{
    \node[draw,circle,inner sep=2,fill=black] (c\i) at ({\diameter*cos(360*\i/5-\rotation)},{(\diameter/1.5)*sin(360*\i/5-\rotation)}) {};
}

\end{tikzpicture}
    \caption{\small }
    \label{figure_locnotglo_a}
\end{subfigure}
\begin{subfigure}{.22\textwidth}
    \centering
\begin{tikzpicture}
\tikzset{highlight/.style={preaction={%But before that
draw,mycolor,-,% Draw yellow without any arrow head
double=mycolor,
double distance=10\pgflinewidth,opacity=0.3
}}}

\def \rotation {360/5+180};
\def \diameter {1};

\foreach \i [count=\xi from 0, remember=\i as \iprev] in {0,...,3} {
    \node[draw,circle,inner sep=2] (\i) at ({\diameter*cos(360*\i/4-\rotation)},{(\diameter/1.5)*sin(360*\i/4-\rotation)}) {};
    \ifthenelse{\i>0}{
        \draw[-] (\iprev) -- (\i);}{}
        }

\foreach \i [count=\xi from 0, remember=\i as \iprev] in {0,...,3}{
  \node[draw,circle,inner sep=2, fill=white] (a\i) at ({\diameter*cos(360*\i/4-\rotation)},{(\diameter/1.5)*sin(360*\i/4-\rotation)+0.7}) {};

  \draw[-] (\i) -- (a\i);

  \ifthenelse{\i>0}{
    \draw[-] (a\iprev) -- (a\i);
  }{}

}

\draw[-] (0) -- (3);
\draw[-] (a0) -- (a3);

\foreach \i in {0,2}{
    \node[draw,circle,inner sep=2, fill=black] (b\i) at ({\diameter*cos(360*\i/4-\rotation)},{(\diameter/1.5)*sin(360*\i/4-\rotation)+0.7}) {};
}

\foreach \i in {1,3}{
    \node[draw,circle,inner sep=2,fill=black] (c\i) at ({\diameter*cos(360*\i/4-\rotation)},{(\diameter/1.5)*sin(360*\i/4-\rotation)}) {};
}

\end{tikzpicture}
    \caption{\small }
    \label{figure_locnotglo_c}
\end{subfigure}
\begin{subfigure}{.22\textwidth}
    \centering
\begin{tikzpicture}
\tikzset{highlight/.style={preaction={%But before that
draw,mycolor,-,% Draw yellow without any arrow head
double=mycolor,
double distance=10\pgflinewidth,opacity=0.3
}}}

\def \rotation {360/5+180};
\def \diameter {1};

\foreach \i [count=\xi from 0, remember=\i as \iprev] in {0,...,3} {
    \node[draw,circle,inner sep=2] (\i) at ({\diameter*cos(360*\i/4-\rotation)},{(\diameter/1.5)*sin(360*\i/4-\rotation)}) {};
    \ifthenelse{\i>0}{
        \draw[-] (\iprev) -- (\i);}{}
        }

\foreach \i [count=\xi from 0, remember=\i as \iprev] in {0,...,3}{
  \node[draw,circle,inner sep=2, fill=white] (a\i) at ({\diameter*cos(360*\i/4-\rotation)},{(\diameter/1.5)*sin(360*\i/4-\rotation)+0.7}) {};

  \draw[-] (\i) -- (a\i);

  \ifthenelse{\i>0}{
    \draw[-] (a\iprev) -- (a\i);
  }{}

}

\draw[-] (0) -- (3);
\draw[-] (a0) -- (a3);

\foreach \i in {0,1}{
    \node[draw,circle,inner sep=2, fill=black] (b\i) at ({\diameter*cos(360*\i/4-\rotation)},{(\diameter/1.5)*sin(360*\i/4-\rotation)+0.7}) {};
}

\foreach \i in {2,3}{
    \node[draw,circle,inner sep=2,fill=black] (c\i) at ({\diameter*cos(360*\i/4-\rotation)},{(\diameter/1.5)*sin(360*\i/4-\rotation)}) {};
}

\end{tikzpicture}
    \caption{\small }
    \label{figure_locnotglo_d}
\end{subfigure}
\caption{\small \small The shaded vertices indicate a global maximizer of $W$ (A), a local maximizer of $W$ that is not global (B), a global minimizer of $E_g$ (C) and a local minimizer of $E_g$ that is not global (D), where $g(r)=\frac{1}{r}$ (see Example \ref{example_locnotglo}).}\label{figure_locnotglo}
\end{figure}
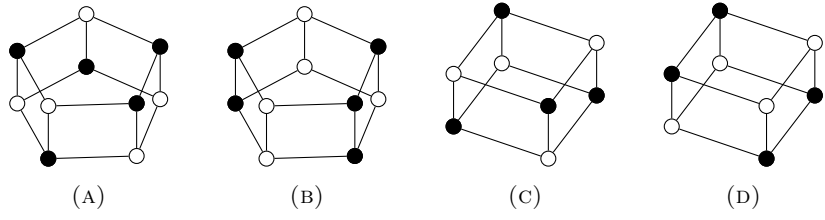 \begin{example}\label{example_locnotglo}
Let $A$ and $B$ be the sets of vertices of the cylinder graph $G=P_2\, \Box\, C_5$ depicted in Figure \ref{figure_locnotglo_b} and Figure \ref{figure_locnotglo_a}, respectively. One can readily check that $A$ is a global maximizer of $W$ with $W(A)=21$ and $B$ is a local maximizer of $W$ with $W(B)=20$. Thus, $B$ is a local maximizer of $W$ that is not a global maximizer of $W$. Let $C$ and $D$ be the sets of vertices of the cube graph $Q_8$ depicted in Figure \ref{figure_locnotglo_c} and Figure \ref{figure_locnotglo_d}, respectively. Then, taking $g(r)=\frac{1}{r}$, one can check that $C$ is a global minimizer of $E_g$ with $E_g(G)=3$ and $D$ is a local minimizer of $E_g$ with $E_g(D)=\frac{11}{3}$. Hence $D$ is a local minimizer of $E_g$ that is not a global minimzer of $E_g$.
\end{example}

Let us point out that the usual notion of \emph{local search} provides a relatively fast way to find local minimizers (maximizers) of $F$. For example, we define an algorithm $\descent$, which takes as input a finite simple connected graph $G$, a real-valued function $F$ defined on the powerset of the vertices of $G$ together with a set of vertices $A$ in $G$, and outputs a local minimizer of $F$ of cardinality $|A|$ after finitely many steps.

\bigskip

\noindent $\descent(G,F,A)$
\begin{enumerate}
\item Let $X=A$.
\item\label{item_list_perts} Let $\vec{L}$ be the list of all perturbations of $X$ in $G$.
\item If $F(\vec{L}(i))<F(X)$ for some $i<\textrm{length}(\vec{L})$ then set $X=\vec{L}(j)$ where $j$ is the least such $i$ and go to (\ref{item_list_perts}).
\item Otherwise, if $F(\vec{L}(i))\geq F(X)$ for all $i<\textrm{length}(\vec{L})$ then return $X$.
\end{enumerate}

\bigskip

\noindent An algorithm $\ascent$ for finding local maximizers of such functions $F$ can be defined in a similar maner.

In Section \ref{section_minimizers_on_cycles}, we provide a new proof of a slight modification of a result due to Douthett and Krantz, which states that for any strictly decreasing strictly convex function $g:\{1,\ldots,\diam(C_n)\}\to\R$ a set of vertices $A$ of an $n$-cycle $C_n$ is a minimizer of $E_g$ if and only if it is \emph{maximally even}\footnote{Let us note that maximally even sets have become popularized due to Toussaint's discovery \cite{MR2212108} that such sets, which Toussint called \emph{Euclidean rhythms}, manifest as rhythms in musical traditions from all over the world; it is worth noting that Toussaint's work has led to the widespread use of \emph{Euclidean rhythm generators} by composers and electronic music producers in many musical genres. As an additional musical connection with the present work, let us note that the first and second author, together with Luke Freeman and Tobias Whitaker \cite{BCFW}, found that certain minimizers of $E_g$ on M\"obius ladder graphs correspond to well-known rhythms in one of Puerto Rico's oldest musical traditions of African origin called bomba. For example, the shaded set of vertices depicted in Figure \ref{figure_bomba_yuba} corresponds to the bomba yuba and that depicted in Figure \ref{figure_bomba_sica} corresponds to the bomba sic\'a (see \cite[Example 3]{BCFW} for details). 
} (see Section \ref{section_me}), as defined in Clough and Douthett's work on music theory. Our proof uses the powerful notion of \emph{majorization} (see Section \ref{section_majorization}), which was introduced independently by Schur \cite{MR0462892, Schur1923} and by Hardy, Littlewood and P\'{o}lya \cite{HLP}. Let us note that the majorization order was previously applied in graph theory by several authors. Folkman and Fulkerson \cite{MR262112} (and others \cite{MR1782037, MR1771389}) used the majorization order on tuples of cardinalities of color classes to study the \emph{feasibility} of certain edge colorings. Dahl \cite{MR2364586} used the majorization order on tuples of distances of trees to define the \emph{majorization-center} of a tree.

Clough and Douthett established many salient properties of maximally even subsets of $C_n$. For example, they proved that the complement of a maximally even subset of $C_n$ is again maximally even (see Theorem \ref{theorem_complements_cd} or \cite[Theorem 3.3]{CloughDouthett}). In Section \ref{section_complements}, we generalize Clough and Douthett's result by proving that the following three conditions are equivalent:
\begin{enumerate}
\item $G$ is \emph{distance degree regular} (see Definition \ref{definition_ddr}).
\item For every function $g:\{1,\ldots,\diam(G)\}\to\R$, the complement of every minimizer of $E_g$ is again a minimizer of $E_g$. 
\item For some function $g:\{1,\ldots,\diam(G)\}\to \R$ such that the set of real numbers $\{g(1),\ldots,g(\diam(G))\}$ is linearly independent over $\Q$, the complement of every minimizer of $E_g$ is a minimizer of $E_g$.
\end{enumerate}
In fact, one can replace the word ``minimizer'' can be changed to ``local minimizer,'' ``maximizer'' or ``local maximizer'' and the equivalence still holds. In particular, since the M\"obius ladder on $2n$ vertices is distance degree regular, the complement of any minimizer of $E_g$ on $M_{2n}$ will again be a minimizer of $E_g$ (see Figure \ref{figure_mobius}).

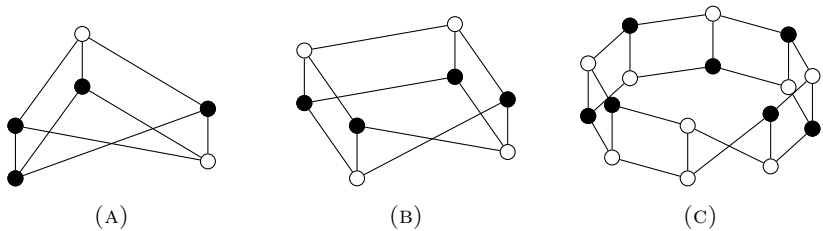
\begin{figure}
\centering
\begin{subfigure}{.3\textwidth}
    \centering
\begin{tikzpicture}
\tikzset{highlight/.style={preaction={%But before that
draw,mycolor,-,% Draw yellow without any arrow head
double=mycolor,
double distance=10\pgflinewidth,opacity=0.3
}}}

\foreach \i [count=\xi from 0, remember=\i as \iprev] in {0,...,2} {
    \node[draw,circle,inner sep=2] (\i) at ({1.5*cos(360*\i/3-360/18)},{0.75*sin(360*\i/3-360/18)}) {};
    \ifthenelse{\i>0}{
        \draw[-] (\iprev) -- (\i);}{}
        }

\foreach \i [count=\xi from 0, remember=\i as \iprev] in {0,...,2}{
  \node[draw,circle,inner sep=2, fill=white] (a\i) at ({1.5*cos(360*\i/3-360/18)},{0.75*sin(360*\i/3-360/18)+0.7}) {};

  \draw[-] (\i) -- (a\i);

  \ifthenelse{\i>0}{
    \draw[-] (a\iprev) -- (a\i);
  }{}

}

\draw[-] (0) -- (a2);
\draw[-] (a0) -- (2);

\foreach \i in {0,2}{
    \node[draw,circle,inner sep=2, fill=black] (b\i) at ({1.5*cos(360*\i/3-360/18)},{0.75*sin(360*\i/3-360/18)+0.7}) {};
}

\foreach \i in {1,2}{
    \node[draw,circle,inner sep=2,fill=black] (c\i) at ({1.5*cos(360*\i/3-360/18)},{0.75*sin(360*\i/3-360/18)}) {};
}

\end{tikzpicture}
    \caption{\small }
    \label{figure_bomba_yuba}
\end{subfigure}
\begin{subfigure}{.3\textwidth}
    \centering
\begin{tikzpicture}
\tikzset{highlight/.style={preaction={%But before that
draw,mycolor,-,% Draw yellow without any arrow head
double=mycolor,
double distance=10\pgflinewidth,opacity=0.3
}}}

\foreach \i [count=\xi from 0, remember=\i as \iprev] in {0,...,3} {
    \node[draw,circle,inner sep=2] (\i) at ({1.5*cos(360*\i/4-360/14)},{0.75*sin(360*\i/4-360/14)}) {};
    \ifthenelse{\i>0}{
        \draw[-] (\iprev) -- (\i);}{}
        }

\foreach \i [count=\xi from 0, remember=\i as \iprev] in {0,...,3}{
  \node[draw,circle,inner sep=2, fill=white] (a\i) at ({1.5*cos(360*\i/4-360/14)},{0.75*sin(360*\i/4-360/14)+0.7}) {};

  \draw[-] (\i) -- (a\i);

  \ifthenelse{\i>0}{
    \draw[-] (a\iprev) -- (a\i);
  }{}

}

\draw[-] (0) -- (a3);
\draw[-] (a0) -- (3);

\foreach \i in {0,3}{
    \node[draw,circle,inner sep=2, fill=black] (b\i) at ({1.5*cos(360*\i/4-360/14)},{0.75*sin(360*\i/4-360/14)+0.7}) {};
}

\foreach \i in {1,2}{
    \node[draw,circle,inner sep=2,fill=black] (c\i) at ({1.5*cos(360*\i/4-360/14)},{0.75*sin(360*\i/4-360/14)}) {};
}

\end{tikzpicture}
    \caption{\small }
    \label{figure_bomba_sica}
\end{subfigure}
\begin{subfigure}{.3\textwidth}
    \centering
\begin{tikzpicture}

\foreach \i [count=\xi from 0, remember=\i as \iprev] in {0,...,7} {
    \node[draw,circle,inner sep=2] (\i) at ({1.5*cos(360*\i/8-360/7)},{0.75*sin(360*\i/8-360/7)}) {};
    \ifthenelse{\i>0}{
        \draw[-] (\iprev) -- (\i);}{}
        }

\foreach \i [count=\xi from 0, remember=\i as \iprev] in {0,...,7}{
  \node[draw,circle,inner sep=2, fill=white] (a\i) at ({1.5*cos(360*\i/8-360/7)},{0.75*sin(360*\i/8-360/7)+0.7}) {};

  \draw[-] (\i) -- (a\i);

  \ifthenelse{\i>0}{
    \draw[-] (a\iprev) -- (a\i);
  }{}

}

\draw[-] (0) -- (a7);
\draw[-] (a0) -- (7);

\foreach \i in {0,2,4,6}{
    \node[draw,circle,inner sep=2, fill=black] (b\i) at ({1.5*cos(360*\i/8-360/7)},{0.75*sin(360*\i/8-360/7)+0.7}) {};
}

\foreach \i in {1,3,5}{
    \node[draw,circle,inner sep=2,fill=black] (c\i) at ({1.5*cos(360*\i/8-360/7)},{0.75*sin(360*\i/8-360/7)}) {};
}

\end{tikzpicture}
    \caption{\small }
    \label{}
\end{subfigure}
\caption{\small \small Minimizers of $E_g$ on various M\"obius ladders where $g(r)=\frac{1}{r}$.}\label{figure_mobius}
\end{figure}

In Section \ref{section_wiener}, we turn our attention to the special case of $E_g$ in which $g(r)=r$, namely, the Wiener index. As mentioned above, in this case, rather than writing $E_g(A)$, where $A$ is a set of vertices, we will write
\[W(A)=\sum_{\{u,v\}\in\binom{A}{2}}d(u,v).\]
In Proposition \ref{proposition_wiener_path}, we provide a characterization of the maximizers of $W$ and the local maximizers of $W$ on paths, which are, in fact, one and the same. In Theorem \ref{theorem_w}, on the cycle $C_n$, we show that the maximizers of $W$ coinside with the local maximizers of $W$, and provide a characterization of such maximizers in terms of the chromatic and diatonic distances studied by Clough and Douthett. We note that, as a consequence of these results, on paths and cycles, the algorithm $\ascent$, defined above, produces maximizers of $W$ (see Corollary \ref{corollary_ascent_on_paths} and Corollary \ref{corollary_ascent_on_cycles}). Finally, generalizing a result of Jiang \cite{MR2394474}, in Theorem \ref{theorem_balanced}, we characterize the maximizers of the Wiener index on $C_n$ using two balanced conditions. Let us note that the notion of \emph{balanced} set of vertices in $C_n$ (Definition \ref{definition_balanced}(1)), which is a straightforward generalization of a notion used by Jiang, is not sufficient for characterizing maximizers of $W$ on $C_n$ (this follows from Theorem \ref{theorem_balanced}), and for this reason we introduce the notion of \emph{weakly balanced} set of vertices in $C_n$ (Definition \ref{definition_balanced}(2)). Let us emphasize that the results of Section \ref{section_wiener}, together with Theorem \ref{theorem_min_energy_cycles}, establishes that for cycles, maximizing with Wiener index $W$ is not the same as minimizing the Harray index $H$: for example, there are many Wiener index maximizers of W on $C_7$ and $C_8$ (see Figure \ref{figure_WI_max_c7} and Figure \ref{figure_WI_max_c8}) that are not maximially even sets.

In Section \ref{section_future}, we discuss a few open questions.

\section{Minimizers on cycles}\label{section_minimizers_on_cycles}

\subsection{Maximal evenness and J-representations}\label{section_me}

Suppose $u$ and $v$ are vertices in $C_n$. Following the convention used in \cite{CloughDouthett,MR2512671,MR2366388,MR2408358,MR1643317}, we will typically assume that the vertices of $C_n$ are arranged around a circle in clockwise order (see Figure \ref{figure_ME_set}). The \emph{clockwise distance} from $u$ to $v$, denoted by $d^*(u,v)$, is the least non-negative integer congruent to $v-u$ modulo $n$. As defined in Section \ref{section_intro}, the \emph{distance} (or the \emph{geodesic distance}) from $u$ to $v$, denoted by $d(u,v)$ is the length of the shortest path from $u$ to $v$. The unique clockwise path from $u$ to $v$ in $C_n$ is denoted by $\cw{C_n}{u}{v}$ and is sometimes referred to as \emph{the interval from $u$ to $v$}. The \emph{chromatic length} of the interval $\cw{C_n}{u}{v}$ is defined to be $d^*(u,v)$. Notice that $|\cw{C_n}{u}{v}|=d^*(u,v)+1$.

For any set $A=\{a_0,\ldots,a_{m-1}\}$ of vertices in $C_n$ with $|A|=m\leq n$, arranged in increasing order, and for $i,j\in\{0,\ldots,m-1\}$, we let $\cw{A}{a_i}{a_j}$ be the collection of elements of $A$ enumerated in clockwise order from $a_i$ to $a_j$; that is,
\[\cw{A}{a_i}{a_j}=\{a_{i+\ell}\st 0\leq\ell\leq d_{C_m}^*(i,j)\},\]
where we assume indices are reduced modulo $m$. Notice that, when we take $A=C_n$, this notation reduces to the usual clockwise interval from $a_i$ to $a_j$ in $C_n$. For $a_i,a_j\in A$, the \emph{diatonic length} of the inteval $\cw{A}{a_i}{a_j}$, also referred to as the \emph{$A$-span} of the ordered pair $(a_i,a_j)$ and written as $\spn_A(a_i,a_j)$, is the least positive integer congruent to $j-i$ modulo $m$. Notice that $|\cw{A}{a_i}{a_j}|=\spn_A(a_i,a_j)+1$. We define two multisets: the \emph{$k$-multispectrum of clockwise distances} of $A$ is the multiset
\[\sigma^*_k(A)=\left[\,d^*(u,v)\st \text{$u,v\in A$, $u\neq v$ \text{ and }$\spn_A(u,v)=k\,$}\right]\]
and the \emph{$k$-multispectrum of geodesic distances} of $A$ is the multiset
\[\sigma_k(A)=\left[\,d(u,v)\st\text{$u,v\in A$, $u\neq v$ \text{ and }$\spn_A(u,v)=k\,$}\right].\]
For a multiset $X$ we define $\supp(X)$ to be the set whose elements are those of $X$. %So, for example, if $X=[\, 1,1,2,2,2,3\, ]$ then $\supp(X)=\{1,2,3\}$.

\begin{definition}[{Clough and Douthett \cite{CloughDouthett}}]\label{definition_me} A set $A$ of vertices in $C_n$ with $|A|=m$ is \emph{maximally even} if for each integer $k$ with $1\leq k \leq m-1$ the set $\supp(\sigma_k^*(A))$ consists of either a single integer or two consecutive integers.
\end{definition}

If we are given a particular set $A$ of vertices in $C_n$, for some particular $n$, we can easily check whether or not $A$ is maximally even using Definition \ref{definition_me}. However, if we are given integers $n$ and $m$ with $0\leq m\leq n$, it is not immediately clear, especially for large values of $n$ and $m$, which sets of vertices in $C_n$ are maximally even. For this reason, Clough and Douthett gave a more explicit characterization of maximal evenness in terms of what they called $J$-representations.

\begin{definition}[{Clough and Douthett \cite{CloughDouthett}}]
Suppose $m$, $n$ and $r$ are integers with $1\leq m\leq n$ and $0\leq r\leq n-1$. We define
\[J^r_{n,m}=\left\{\floor{\frac{ni+r}{m}}\st i\in \Z \text{ and } 0\leq i < m\right\}.\]
Sets of the form $J^r_{n,m}$ are called \emph{$J$-representations}.
\end{definition}

\begin{theorem}[{Clough and Douthett \cite[Theorem 1.2 and Theorem 1.5]{CloughDouthett}}]\label{theorem_J_is_maximally_even}
Suppose $m$ and $n$ are integers with $1\leq m\leq n$. Then a set of vertices is maximally even in $C_n$ if and only if it is of the form $J^r_{n,m}$ for some integer $r$ with $0\leq r\leq n$.
\end{theorem}

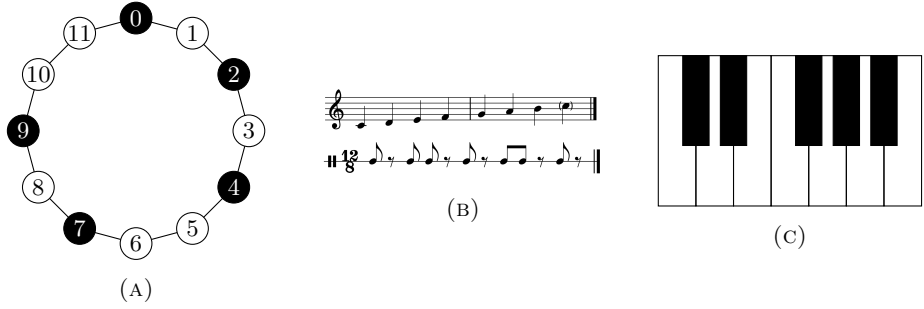
\begin{figure}
\begin{subfigure}[c]{.3\textwidth}
    \centering
\begin{tikzpicture}
  % Define vertices with selective filling
  \foreach \i in {0,2,4,7,9} {
    \coordinate (v\i) at ({90-(360/12)*\i}:1.5);
    \node[draw, circle, fill=black, text=white, inner sep=0pt, minimum width = 12pt] at (v\i) {\small $\i$};
  }
  
  % Define the rest of the vertices without filling
  \foreach \i in {1,3,5,6,8,10,11} {
    \coordinate (v\i) at ({90-(360/12)*\i}:1.5);
    \node[draw, circle, fill=white, inner sep=0pt, minimum width = 12pt] at (v\i) {\small $\i$};
  }
  
  % Draw edges in the background layer
  \begin{pgfonlayer}{background}
    \foreach \i in {0,...,10} {
      \draw (v\i) -- (v\the\numexpr\i+1\relax);
    }
    \draw (v11) -- (v0); % Closing edge
  \end{pgfonlayer}
\end{tikzpicture}
\caption{\small }
\end{subfigure}
\hspace{0.025\textwidth}
\begin{subfigure}[c]{.3\textwidth}
    \centering
\vfill
\includegraphics[width=\textwidth]{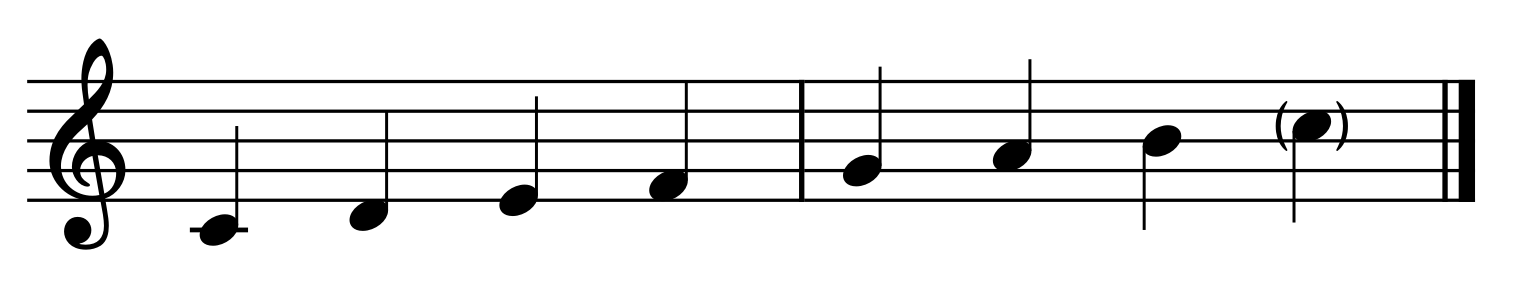}
\vfill
\includegraphics[width=\textwidth]{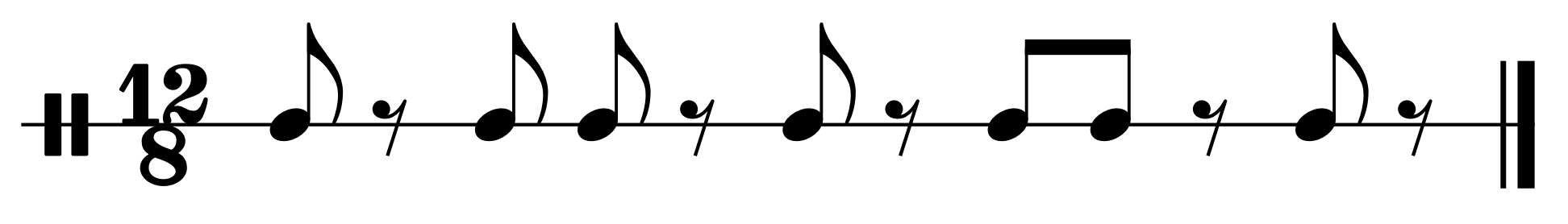}
\vfill
\caption{\small }
\end{subfigure}
\hspace{0.025\textwidth}
\begin{subfigure}[c]{0.3\textwidth}
    \centering
\vfill
\begin{tikzpicture}[scale=0.5]
    % Define which keys to highlight (1 for highlight, 0 for normal)
    % Format: {white keys}{black keys}
    \def\highlightWhiteKeys{{0,0,0,0,0,0,0}}
    \def\highlightBlackKeys{{0,0,0,0,0,0}}%oops, extra 0 first
    \def\WkeyLabels{{0,2,4,5,7,9,11}} % Key labels
    \def\BkeyLabels{{1,3,0,6,8,10}}
    % Drawing white keys and highlighting
    \foreach \i in {0,...,6} {
        \pgfmathparse{\highlightWhiteKeys[\i]}
        \ifnum\pgfmathresult=1
            \fill[Pinkish] (\i,0) rectangle (\i+1,4); % Highlight key
        \fi
        \draw (\i,0) rectangle (\i+1,4); % Draw key
        \pgfmathsetmacro{\label}{\WkeyLabels[\i]}
        %\node at (\i+0.5,0.5) [circle,draw,inner sep=2pt,fill=white,minimum width = 14pt] {\scriptsize \label};
    }

    % Drawing black keys and highlighting
    \foreach \i [count=\j] in {0,1,3,4,5} {
        \pgfmathparse{\highlightBlackKeys[\j]}
        \ifnum\pgfmathresult=1
            \draw[fill=Pinkish!70!black] (\i+0.65,1.61) rectangle (\i+1.35,4); % Highlight key
        \fi
        \ifnum\pgfmathresult=0
            \draw[fill=black] (\i+0.65,1.61) rectangle (\i+1.35,4);
        \fi
        \pgfmathsetmacro{\label}{\BkeyLabels[\i]}
        %\node at (\i+1,2) [circle,draw,inner sep=2pt,fill=white,minimum width = 14pt] {\scriptsize \label};
    }
\end{tikzpicture}
\vfill
\caption{\small }
\end{subfigure}
\caption{\small \small The maximally even set $J^0_{12,5}$ is shown in black (A), a rotation of the complement of $J^0_{12,5}$ is written in standard notation as both a musical scale and rhythm (B) and $J^0_{12,5}$ as well as its complement can be visualized as the black keys and white keys, respectively, on a piano keyboard.}\label{figure_ME_set}
\end{figure}

\begin{example}
If the vertices of $C_{12}$ are viewed as representing the 12 standard pitch classes, the set $A=J^0_{12,5}=\{a_0,a_1,a_2,a_3,a_4\}=\{0,2,4,7,9\}$, shown in Figure \ref{figure_ME_set}, can be identified with the pitch classes corresponding to the black keys on a piano keyboard. Then the complement $V(C_n)\setminus J^0_{12,5}=J^{11}_{12,7}$ corresponds to the diatonic scale, i.e., the white keys on a piano keyboard. Notice that 
\[\spec_1(A)=\sigma_1(A)=[\,2,2,2,3,3\,]\]
and
\[\spec_2(A)=\sigma_2(A)=[\,4,5,5,5,5\,],\]
and for $k\in\{3,4\}$, we have $\spec_3(A)=[\,7,7,7,7,8\,]$, $\sigma_3(A)=[\,4,5,5,5,5\,]$, $\spec_4(A)=[\,9,9,10,10,10\,]$ and 
$\sigma_4(A)=[\,2,2,2,3,3\,].$ Let us note that, it is not a coincidence that in this example there is a redundancy in the multispectra $\sigma_k(A)$ of geodesic distances in the sense that $\sigma_3(A)=\sigma_2(A)$ and $\sigma_4(A)=\sigma_1(A)$. Indeed, Corollary \ref{corollary_spec_characterization} explicitly states that we can characterize maximally even sets by considering the sets $\sigma_k(A)$ just for $1\leq k\leq\floor{\frac{m}{2}}$.
\end{example}

For the reader's convenience, let us collect some straight-forward corollaries of results of Clough and Douthett. Let us emphasize that the corollaries below differ from the results of Clough and Douthett in that they involve geodesic distances rather than clockwise distances.

\begin{corollary}\label{corollary_spec_of_J_rep}
Suppose $k$, $m$, $n$ and $r$ are integers with $1\leq m\leq n$ and $0\leq r\leq n-1$. 
\begin{enumerate}[\normalfont(1)]
\item If $0<k<\frac{m}{2}$ or $n$ is even then 
\[\supp(\sigma_k(J^r_{n,m}))=\left\{\floor{\frac{nk}{m}},\ceil{\frac{nk}{m}}\right\}.\]
\item If $k=\frac{m}{2}$ and $n$ is odd then 
\[\supp(\sigma_k(J^r_{n,m}))=\left\{\floor{\frac{n}{2}}\right\}.\]
\end{enumerate}
\end{corollary}

\begin{corollary}\label{corollary_sum_of_spec_of_J_rep}
Suppose $k,m,n$ and $r$ are integers such that $1\leq m\leq n$ and $1\leq k\leq\floor{\frac{m}{2}}$. Then
\[\sum\sigma_k(J^r_{n,m})=\begin{cases}
nk & \text{if $k<\frac{m}{2}$ or $n$ is even}\\
(n-1)k & \text{if $k=\frac{m}{2}$ and $n$ is odd.}
\end{cases}\]
\end{corollary}

\begin{corollary}\label{corollary_spec_characterization}
Suppose $m$ and $n$ are integers with $1\leq m\leq n$. Then a set $A$ of vertices in $C_n$ is maximially even if and only if the following conditions hold.
\begin{enumerate}[\normalfont(1)]
\item If $0<k<\frac{m}{2}$ or $n$ is even then 
\[\supp(\sigma_k(A))=\left\{\floor{\frac{nk}{m}},\ceil{\frac{nk}{m}}\right\}.\]
\item If $k=\frac{m}{2}$ and $n$ is odd then 
\[\supp(\sigma_k(A))=\left\{\floor{\frac{n}{2}}\right\}.\]
\end{enumerate}
\end{corollary}

The following lemma is folklore and is implicit in \cite{MR1401228} and \cite[Lemma 1.9]{CloughDouthett}.

\begin{lemma}\label{lemma_consecutive}
Suppose $m$ and $S$ are integers with $m\geq 1$. There exists a unique multiset of integers $D(m,S)=\left[\,x_1,\ldots,x_m\,\right]$ of cardinality $m$ such that 
\begin{enumerate}[\normalfont(1)]
\item $\displaystyle\sum_{i=1}^mx_i=S$ and
\item $\displaystyle\max_{1\leq i\leq m}x_i-\min_{1\leq i\leq m} x_i\leq 1$.
\end{enumerate}
Furthermore, we have $\supp(D(m,S))=\left\{\floor{\frac{S}{m}},\ceil{\frac{S}{m}}\right\}$. Moreover, if $\max(x)-\min(x)=1$, then the number of occurrences of $\floor{\frac{S}{m}}$ in $D(m,S)$ is 
$m(\floor{\frac{S}{m}}+1)-S$ and the number of occurrences of $\ceil{\frac{S}{m}}$ is 
$S-m\floor{\frac{S}{m}}$.
\end{lemma}

%As we will need them in what follows, let us state several lemmas used by Clough and Douthett to establish that a set of vertices in $C_n$ is maximally even if and only if it has a $J$-representation.

\subsection{Majorization}\label{section_majorization}
Let us define several variants of the \emph{majorization} order on tuples of real numbers. For background and more details on majorization, see \cite{MR2759813}. Suppose $x$ is a finite tuple $(x_1,\ldots,x_n)$ of real numbers. Let $x_{(1)}\leq \cdots\leq x_{(n)}$ denote the components of $x$ in non-decreasing order and let $x^{\uparrow}=(x_{(1)},\ldots,x_{(n)}).$
Similarly, let $x_{[1]}\geq\cdots\geq x_{[n]}$
denote the components of $x$ in non-increasing order and we let $x^{\downarrow}=(x_{[1]},\ldots,x_{[n]})$.

%\begin{definition}
%For $x,y\in\R^n$ we define 
%\[x\maj y \ \ \text{if}\ \ \begin{cases}
%\displaystyle\sum_{i=1}^kx_{[i]}\leq\sum_{i=1}^ky_{[i]}$ for $k=1,\ldots,n-1,\text{ and}\vspace{3mm}\\
%\displaystyle\sum_{k=1}^nx_{[i]}=\sum_{k=1}^ny_{[i]}.
%\end{cases}\]
%\end{definition}

\begin{definition}\label{definition_maj} Suppose $x,y\in\R^n$. We define 
\[x\maj y \ \ \text{if}\ \ \begin{cases}\displaystyle\sum_{i=1}^kx_{[i]}\leq\sum_{i=1}^ky_{[i]}$ for $k=1,\ldots,n-1,\text{ and}\vspace{3mm}\\ \displaystyle\sum_{k=1}^nx_{[i]}=\sum_{k=1}^ny_{[i]}.\end{cases}\] 
When $x\maj y$ we say that $x$ is \emph{majorized} by $y$, or that $y$ \emph{majorizes} $x$. Similarly, we define 
\[x\submaj y \ \ \text{if}\ \ \sum_{i=1}^kx_{[i]}\leq\sum_{i=1}^ky_{[i]} \text{ for } k=1,\ldots,n,\] and when $x\submaj y$ we say that $x$ is \emph{weakly submajorized} by $y$, or that $y$ \emph{weakly submajorizes} $x$. Furthermore, we let \[x\supermaj y \ \ \text{if}\ \ \sum_{i=1}^kx_{(i)}\geq\sum_{i=1}^ky_{(i)} \text{ for } k=1,\ldots,n,\] and when $x\supermaj y$ we say that $x$ is \emph{weakly supermajorized} by $y$, or that $y$ \emph{weakly supermajorizes} $x$. \end{definition}

\begin{remark}
Let us note that the notation $x\maj y$, for tuples $x,y\in\R^n$, may be somewhat counterintuitive because $\maj$ is not a strict ordering since $x\maj x$ for all $x\in\R^n$. Furthermore, the notation $x\supermaj y$ and $x\submaj y$ may also be misleading because, for example $(5,5)\supermaj (4,4)$. Nonetheless, this notation is widely used (see \cite{MR2759813}), and so we adopt it here.
\end{remark}

\begin{remark}
For a \emph{multiset} $x=[x_1,\ldots,x_n]$ of real numbers with $n$ elements, as above, we let $x^\uparrow=(x_{(1)},\ldots,x_{(n)})$ and $x^{\downarrow}=(x_{[1]},\ldots,x_{[n]})$. Notice that for two multisets $x$ and $y$ with $n$ elements, the definitions of $x\maj y$, $x\submaj y$ and $x\supermaj y$ make sense as stated in Definition \ref{definition_maj}, because these notions are defined for tuples without mentioning the particular indexing. In what follows, for such multisets of real numbers we will use the notation $x\maj y$, $x\submaj y$ and $x\supermaj y$ without further comment.
\end{remark}

Classical results of Schur \cite{Schur1923, MR0462892} (see \cite[Proposition I.3.C.1.a]{MR2759813}), Hardy Littlewood and P\'olya \cite{HLP} (see \cite[Proposition I.4.B.1]{MR2759813}) and Tomi\'c \cite{MR0031006} (see \cite[Proposition I.4.B.2]{MR2759813}), established the first connections between the majorization order on tuples and certain sums obtained from convex functions. From these results, we easily obtain the following, where for a set of vertices $A$ in a graph $G$ we let $D(A)$ be the \emph{multiset} of all pairwise distances between vertices in $A$, that is,
\[D(A)=\left[\,d(u,v)\st \{u,v\}\in \binom{A}{2}\,\right].\] 

\begin{corollary}\label{corollary_maj_consistent_with_energy}
Suppose $G$ is a finite simple connected graph and $A$ and $B$ are sets of vertices in $G$ with $|A|=|B|$.
\begin{enumerate}[\normalfont(1)]
\item $D(A)\maj D(B)$ if and only if $E_g(A)\leq E_g(B)$ for all convex functions $g:\{1,\ldots,\diam(G)\}\to \R$.
\item $D(A)\submaj D(B)$ if and only if $E_g(A)\leq E_g(B)$ for all increasing convex functions $g:\{1,\ldots,\diam(G)\}\to\R$.
\item $D(A)\supermaj D(B)$ if and only if $E_g(A)\leq E_g(B)$ for all decreasing convex functions $g:\{1,\ldots,\diam(G)\}\to\R$.
\end{enumerate}
\end{corollary}
Additionally applying \cite[Theorem I.3.A.8.a]{MR2759813} we obtain the following corollary.
\begin{corollary}\label{corollary_we_need}
Suppose $I\subseteq\R$ is an interval, $g:I\to\R$ is a function and
\[\phi(x)=\sum_{i=1}^n g(x_i).\]
\begin{enumerate}
\item Suppose $g$ is strictly convex on $I$. Then for all $x,y\in I^n$, whenever $x\maj y$ and $x$ is not a permutation of $y$, it follows that $\phi(x)<\phi(y)$.
\item Suppose $g$ is strictly decreasing and strictly convex on $I$. Then for all $x,y\in I^n$, whenever $x\supermaj y$ and $x$ is not a permutation of $y$, it follows that $\phi(x)<\phi(y)$.
\item Suppose $g$ is strictly increasing and strictly convex on $I$. Then for all $x,y\in I^n$, whenever $x\submaj y$ and $x$ is not a permutation of $y$, it follows that $\phi(x)<\phi(y)$.
\end{enumerate}
\end{corollary}

\subsection{Maximal evenness and majorization}

Let us define two operations that we will need to perform on multisets to prove Theorem \ref{theorem_min_energy_cycles}.

\begin{definition} The \emph{up transfer} $\Up$ and the \emph{Robin Hood\footnote{The Robin Hood transfer is so named since, as noted in \cite[Page 7]{MR2759813} it ``mirrors precisely an operation by that `worthy' outlaw in Sherwood Forest.'' The same function is also called a \emph{$T$-transform}, a \emph{Dalton transfer} or a \emph{pinch}.} transfer} $\Rh$ are functions from the collection of all finite multisets of real numbers to itself. We define $\Rh(\emptyset)=\Up(\emptyset)=\emptyset$. Suppose $x$ is a finite multiset of real numbers with $|x|=m\geq 1$.
\begin{enumerate}
\item If $m=1$ then $\Rh(x)=x$. If $m\geq 2$, let 
\[\Rh(x)=[\,x_{(1)}+1,x_{(2)},\ldots,x_{(m-1)},x_{(m)}-1\,].\]
\item Let $\Up(x)=[\,y_1,\ldots,y_m\,]$ where $y_{(1)}=x_{(1)}+1$ and $y_{(i)}=x_{(i)}$ for $1<i\leq m$.
\end{enumerate}
\end{definition}

Recall (see Lemma \ref{lemma_consecutive}) that for integers $S$ and $m$ with $m\geq 1$, $D(m,S)$ is the unique multiset of integers with cardinality $m$ whose sum is $S$ such that no two members of the multiset differ by more than $1$.

\begin{lemma}\label{lemma_simple_operations}
If $S$ is an integer and $x=[x_1,\ldots,x_m]$ is a nonempty multiset of integers with cardinality $m\geq 1$ such that $\sum_{i=1}^mx_i\leq S$, then $D(m,S)\supermaj x$.
\end{lemma}

\begin{proof}
Since $\sum x\leq S$, we may let $\ell_1$ be the least integer $i\geq 0$ such that $\sum\Up^i(x)=S$. Notice that for each integer $i\geq 0$ the multiset $\Rh^i\Up^{\ell_1}(x)$ has cardinality $m$ and sums to $S$. Let $\ell_2$ be the least integer $i\geq 0$ such that $\Rh^{i+1}\Up^{\ell_1}(x)=\Rh^i\Up^{\ell_1}(x)$. Notice that $\Rh^{\ell_2}\Up^{\ell_1}(x)$ must consist of either a single integer or two consecutive integers because otherwise,
\[\max(\Rh^{\ell_2}\Up^{\ell_1}(x))-\min(\Rh^{\ell_2}\Up^{\ell_1}(x))\geq 2\]
and hence $\Rh\Rh^{\ell_2}\Up^{\ell_1}(x)\neq \Rh^{\ell_2}\Up^{\ell_1}(x)$, a contradiction. Therefore, by Lemma \ref{lemma_consecutive}, we obtain $\Rh^{\ell_2}(\Up^{\ell_1}(x))=D(m,S)$. Since $\Up(y)\supermaj y$ for any finite multiset of real numbers $y$ and since $z\supermaj w$ implies $\Rh(z)\supermaj w$ for any finite multisets of real numbers $z$ and $w$ of the same cardinality, it follows that \[D(m,S)=\Rh^{\ell_2}\Up^{\ell_1}(x)\supermaj x.\] \end{proof}

The following theorem is due to Barrett \cite{Barrett} in the case where $g(r)=\frac{1}{r}$ and to Douthett and Krantz (see \cite[Theorem 2]{MR1401228} and \cite[Corollary 3]{MR2408358}) in a modified form wherein the energy of a set of vertices in $C_n$ is taken to be a sum over all ordered pairs and clockwise distances are used (as opposed to geodesic distances), whereas here we sum over all unordered pairs. Although the following theorem was essentailly already known, our proof is new, and the use of the majorization order leads to a simpler presentation which could potentially find applications elsewhere.

\begin{theorem}\label{theorem_min_energy_cycles}Suppose $ g:[1,\diam(C_n)]\to\R$ is strictly decreasing and strictly convex, and let $A$ be a set of vertices of $C_n$ with $|A|=m$. Then $A$ is a minimizer of $E_g$ if and only if it is maximally even.\end{theorem}

\begin{proof}
For the forward direction, suppose $A$ is not maximally even. Then, by Corollary \ref{corollary_spec_characterization}, the collection $K$ of all integers $k$ with $1\leq k\leq \frac{m}{2}$ for which $\sigma_k(A)\neq\sigma_k(J^0_{n,m})$ is nonempty. Suppose $1\leq k\leq \frac{m}{2}$. If $k\notin K$, then $\sigma_k(J^0_{n,m})=\sigma_k(A)$ and hence $E_{g,k}(J^0_{n,m})=E_{g,k}(A)$. Suppose that $k\in K$. Then $\sum\sigma_k(A)\leq S_k$ where $S_k=\sum\sigma_k(J^0_{n,m})$. Since $\sigma_k(J^0_{n,m})$ consists of one or two consecutive integers by Corollary \ref{corollary_spec_of_J_rep}, it follows by Lemma \ref{lemma_consecutive} and Lemma \ref{lemma_simple_operations} that $\sigma_k(J^0_{n,m})=D(m,S_k)\supermaj \sigma_k(A)$. Using the fact that $\sigma_k(A)\neq \sigma_k(J^0_{n,m})$, Corollary \ref{corollary_we_need} implies $\sum g \sigma_k(J^0_{n,m})<\sum g \sigma_k(A)$ and hence $E_{g,k}(J^0_{n,m})<E_{g,k}(A)$. Since $K\neq\emptyset$, 
\[E_g(J^0_{n,m})=\sum_{k=1}^{\ell} E_{g,k}(J^0_{n,m})<\sum_{k=1}^{\ell}E_{g,k}(A)=E_g(A),\]
where $\ell=\floor{\frac{m}{2}}$.

For the reverse direction, suppose that $A$ is maximally even. Then for all integers $k$ with $1\leq k \leq\frac{m}{2}$ we have $\sigma_k(A)= \sigma_k(J^0_{n,m})$. We must show that $A$ is a minimizer of $E$. Let $B\subseteq V(C_n)$ with $|B|=m$ be a minimizer of $E$. By the above argument, it follows that $\sigma_k(B)=\sigma_k(J^0_{n,m})$ for $1\leq k\leq\frac{m}{2}$. Thus, for each integer $k$ with $1\leq k\leq\frac{m}{2}$, the multisets $\sigma_k(A)$ and $\sigma_k(B)$ satisfy the conditions of Lemma \ref{lemma_consecutive}, and hence $\sigma_k(A)=\sigma_k(B)=\sigma_k(J^0_{n,m})$. Therefore, $E_g(A)=E_g(B)$, so $A$ is a minimizer of $E_g$.
\end{proof}

Consider the strictly decreasing strictly convex function $g:(0,\infty)\to\R$ defined by $g(r)=\log\left(\frac{1}{r}\right)$. Let $A$ be a set of vertices in a finite simple connected graph $G$. Since
\[E_g(A)=-\sum_{\{u,v\}\in\binom{A}{2}}\log\left(d(u,v)\right)=-\log\left(\prod_{\{u,v\}\in\binom{A}{2}}d(u,v)\right),\]
it easily follows that $A$ is a minimizer of $E_g$ if and only if $A$ is a maximizer of the function $F$ defined by 
\[F(A)=\prod_{\{u,v\}\in\binom{A}{2}}d(u,v).\]
Thus, we obtain the following corollary of Theorem \ref{theorem_min_energy_cycles}.

\begin{corollary}
A set of vertices $A$ in $C_n$ is maxmially even if and only if $A$ is a maximizer of the function $F$ defined above.
\end{corollary}

\section{Complements of minimizers and maximizers}\label{section_complements}

In this section, building on techniques used by Douthett and Krantz \cite[Theorem 1]{MR2408358} and by Althuis and G\"obel to prove a \emph{generalized hexachordal theorem} \cite[Theorem 1]{hexchord_paper}, we generalize the following theorem of Clough and Douthett which implies that the complement of every maximially even set in a cycle is maximally even.

\begin{theorem}[{Clough and Douthett \cite[Theorem 3.3]{CloughDouthett}}]\label{theorem_complements_cd}
For any positive integers $m\leq n$ and any non-negative integer $r$ with $0\leq r\leq m-1$ we have 
\[C_n\setminus J^r_{n,m}=J^{n-r-1}_{n,n-m}.\]
\end{theorem}
\noindent 

We will show that for a rather wide class of graphs $G$ and functions $g:\{1,\ldots,\diam(G)\}\to\R$, the complement of every minimizer of $E_g$ is a minimizer of $E_g$.

\begin{definition}\label{definition_ddr} Suppose $G$ is a finite simple connected graph and $u$ is a vertex in $G$. We define the \emph{distance vector of $u$ in $G$} to be the tuple $\vec{D}(u)=(d_1,\ldots,d_L)$ of length $L=|V(G)|-1$ consisting of all distances from $u$ to other vertices in $G$, where $d_1\leq \cdots\leq d_L$. The graph $G$ is \emph{distance degree regular} if whenever $u$ and $v$ are vertices in $G$ we have $\vec{D}(u)=\vec{D}(v)$. 
\end{definition}

Let us discuss another characterization of distance degree regular graphs. Given a vertex $u$ in $G$ and $i\in\{0,\ldots,\diam(G)\}$, the \emph{sphere centered at $u$ of radius $i$} is the set
\[S(u,i)=\{v\in V(G)\st d(u,v)=i\}.\]
The \emph{$i$-degree} of a vertex $u$ in $G$ is the number of occurances of $i$ in $\vec{D}(u)$, or equivalently, the cardinality of the set $S(u,i)$. For an integer $\ell$, we say that $G$ is \emph{$\ell$-regular at distance $i$} if all vertices in $G$ have $i$-degree $\ell$, and $G$ is \emph{regular at distance $i$} if it is $\ell$-regular at distance $i$ for some $\ell$. It is not hard to see that $G$ is distance degree regular if and only if it is regular at distance $i$ for all $i\in\{1,\ldots,\diam(G)\}$.

Bloom et al. \cite{MR0634519} introduced distance degree regular graphs and pointed out that every vertex transitive graph is distance degree regular, but the converse does not hold. Indeed, they showed that for every $r\geq 5$ there exist $r$-regular distance degree regular graphs with trivial automorphism groups. Huilgol et al. \cite{MR3059157} proved that the class of distance degree regular graphs is closed under cartesian products. By applying these results, one can produce a wide array of distance degree regular graphs.

Given a multiset $X$, we let $\mult{X}{i}$ be the number of occurances of $i$ in $X$. Suppose $G$ is a finite connected graph. Recall that, for $A\subseteq V(G)$, $D(A)$ is the multiset of distances $d(u,v)$ such that $\{u,v\}\subseteq A$ and $u\neq v$. For $A,B\subseteq V(G)$ we let $D(A,B)$ denote the multiset of distances $d(u,v)$ such that $\{u,v\}\subseteq V(G)$, $u\neq v$, $|\{u,v\}\cap A|=1$ and $|\{u,v\}\cap B|=1$.

\begin{lemma}\label{lemma_circle}
Suppose a graph $G$ is $\ell$-regular at distance $i$ where $i\in\{1,\ldots, \diam(G)\}$. Then
\[\ell|G|=2|\mult{D(V(G))}{i}|.\]
\end{lemma}

\begin{proof}
For any graph $G$ we have  $\sum_{u\in V(G)}|S(u,i)|=2|\mult{D(V(G))}{i}|$ because on the left-hand-side, every pair of vertices that are at distance $i$ from one another is counted exactly twice. Since $G$ is $\ell$-regular at distance $i$ we have $|S(u,i)|=\ell$ for all $u\in V(G)$ and the lemma follows.
\end{proof}

\begin{lemma}\label{lemma_A_mixed_pairs}
Suppose a graph $G$ is $\ell$-regular at distance $i$ where $i\in\{1,\ldots,\diam(G)\}$. For any $A\subseteq V(G)$, we have
\[\mult{D(A,\bar{A})}{i}=\ell|A|-2|\mult{D(A)}{i}|.\]

%the number of $A$-mixed pairs of vertices from $G$ that are at distance $d$ from one another is
%\[|A|\frac{2n_d(V(G))}{n^2}-2n_d(A).\]
\end{lemma}

\begin{proof}
Let $<$ be a linear order on the vertices of $G$. Let $\mathcal{L}$ be the set of all pairs of vertices in $G$ that are at distance $i$ from each other whose $<$-least element is in $A$, and let $\mathcal{G}$ be the set of all pairs of vertices in $G$ that are at distance $i$ from each other whose $<$-greatest element is in $A$. Then symmetric difference $\mathcal{L}\triangle\mathcal{G}=(\mathcal{L}\setminus\mathcal{G})\cup(\mathcal{G}\setminus\mathcal{L})$ is the set of all pairs of vertices in $G$, one of whose elements is in $A$ and one of whose elements is in $\bar{A}$, that are at distance $i$ from each other. We have
\begin{align*}
\mult{D(A,\bar{A})}{i}&=|\mathcal{L}\triangle\mathcal{G}|\\
    &=|\mathcal{L}|+|\mathcal{G}|-2|\mathcal{L}\cap\mathcal{G}|\\
    &=\sum_{u\in A}|S(u,d)| - 2|\P_2(A,d)|\\
    &=\ell |A| - 2|\P_2(A,d)|.
\end{align*}
\end{proof}

%Theorem \ref{theorem_complements} for maximizers and minimizers is obtained as an easy corollary of the following theorem since the quantity on the right hand side of (\ref{equation_complements}) does not depend on which vertices are in $A$. Furthermore, Theorem \ref{theorem_complements} for local maximizers and local minimizers is also easily obtained from the next theorem by additionally noting that, using terminology defined in Section \ref{section_intro}, the complement of a perturbation of a set of vertices is a perturbation of the complement of the set of vertices.

%As we will use it in the next proof, recall that a set of real numbers $\{r_1,\ldots,r_k\}$ is said to be \emph{linearly independent over $\Z$} if whenever $c_1r_1+\cdots+c_kr_k=0$ and $n_i\in\Z$ for $i\in\{1,\ldots,k\}$, it follows that $c_1=\cdots=c_k=0$.

We are ready to provide a characterization of distance degree regular graphs in terms of complements of minimizers being minimizers. In the following theorem, the intended meanings of items (3) and (4) is that one can phrase these conditions in terms of minimizers, local minimizers, maximizers or local maximizers, and in each case one obtains a condition that is equivalent to $G$ being distance degree regular.

\begin{theorem}\label{theorem_complements}
Suppose $G$ is a finite connected graph with diameter $D$. Then the following are equivalent.
\begin{enumerate}
    \item\label{item_comp_ddr} $G$ is distance degree regular.
    \item\label{item_comp_const} For every function $g:\{1,\ldots,D\}\to\mathbb{R}$ and every $A\subseteq V(G)$ with $|A|>0$ we have
\begin{align}E_g(A)-E_g(V(G)\setminus A)=\left(\frac{2|A|}{|G|}-1\right) E_g(V(G)).\label{equation_complements}
\end{align}
    \item \label{item_comp_comp} For every function $g:\{1,\ldots,D\}\to\R$ the complement of every (local) minimizer (maximizer) of $E_g$ is a (local) minimizer (maximizer) of $E_g$.
    \item \label{item_comp_lin_indep} For some function $g:\{1,\ldots,D\}\to\R$ such that the set of real numbers $\{g(1),\ldots,g(D)\}$ is linearly independent over $\Q$, the complement of every (local) minimizer (maximizer) of $E_g$ is a (local) minimizer (maximizer) of $E_g$.
%\begin{align}E_g(A)-E_g(V(G)\setminus A)=\left(\frac{2m}{|G|}-1\right) E_g(V(G)).\label{equation_complements}
%\end{align}
\end{enumerate}
\end{theorem}

\begin{proof}
Suppose $G$ is a finite connected distance degree regular graph. If $A$ is a set of vertices in $G$, we let $\bar{A}=V(G)\setminus A$.

Let us prove (\ref{item_comp_ddr}) $\implies$ (\ref{item_comp_const}). Suppose $A$ is a set of vertices in $G$ with $|A|=m$. For $i\in\{1,\ldots,D\}$, let $m_i=\mult{D(A)}{i}$, $n_i=\mult{D(\bar{A})}{i}$, $M_i=\mult{D(V(G))}{i}$ and suppose $G$ is $\ell_i$-regular at distance $i$. By Lemma \ref{lemma_A_mixed_pairs} we have
\[M_i = m_i+n_i+\ell_i m-2m_i,\]
and after applying Lemma \ref{lemma_circle}, 
\begin{align*}
M_i&=m_i+n_i+\left(\frac{2M_i}{|G|}\right)m-2m_i\\
m_i-n_i&=\left(\frac{2m}{|G|}-1\right)M_i.
\end{align*}
We have
\begin{align*}
E_g(A)- E_g(\bar{A})&=\sum_{i=1}^D m_i g(i)-\sum_{i=1}^D n_i g(i)\\
    &=\sum_{i=1}^D(m_i-n_i)g(i)\\
    &=\sum_{i=1}^D\left(\frac{2m}{|G|}-1\right)M_i g(i)\\
    &=\left(\frac{2m}{|G|}-1\right)\sum_{i=1}^D M_i g(i)\\
\end{align*}

We obtain (\ref{item_comp_comp}) for maximizers and minimizers as an easy consequence of (\ref{item_comp_const}) since the quantity on the right hand side of equation (\ref{equation_complements}) does not depend on which vertices are in $A$. Furthermore, (\ref{item_comp_comp}) for local maximizers and local minimizers is also easily obtained from (\ref{item_comp_const}) by additionally noting that, using terminology defined in Section \ref{section_intro}, the complement of a perturbation of a set of vertices is a perturbation of the complement of the set of vertices.

%Fix a function $g:\{1,\ldots,D\}\to\R$ such that the set of real numbers $\{g(1),\ldots,g(D)\}$ is linearly independent over $\Z$. To show that $G$ is distance degree regular it suffices to show that $G$ is robust. For $A\subseteq V(G)$ with $|A|>0$, notice that by (\ref{equation_complements}), we have \[\sum_{i=1}^D\mult{D(A)}{i}g(i)-\sum_{i=1}^D\mult{D(\bar{A})}{i}g(i)=K_{|A|},\] where $K_{|A|}$ is a constant that depends on $|A|$, $g$ and $G$, but is independent of which particular vertices are in $A$. Suppose $A,B\subseteq V(G)$ with $|A|=|B|=m>0$, then \[\sum_{i=1}^D\delta_A(i)g(i)=\sum_{i=1}^D\delta_B(i)g(i)=K_m.\] Since $\{g(1),\ldots,g(D)\}$ is linearly independent over $\Z$, it follows that $\delta_A(i)=\delta_B(i)$ for all $i\in\{1,\ldots,D\}$, and therefore $G$ is robust.

Clearly, (\ref{item_comp_comp}) $\implies$ (\ref{item_comp_lin_indep}), so it only remains to prove that (\ref{item_comp_lin_indep}) $\implies$ (\ref{item_comp_ddr}). Fix a $g$ such that the set of real numbers $\{g(1),\ldots,g(D)\}$ is linearly independent over $\Q$. For each $v\in V(G)$ we have $E_g(\{v\})=0$ and hence $\{v\}$ is a minimizer of $E_g$. Thus, by assumption for each $v\in V(G)$ the set $\overline{\{v\}}$ is a minimizer of $E_g$ as well. This implies that, there is a fixed real number $K$ with 
\begin{align}
E_g\left(\overline{\{v\}}\right)=\sum_{i=1}^D\mult{D\left(\overline{\{v\}}\right)}{i}g(i)=K\label{equation_lin_comb}
\end{align}
for all $v\in V(G)$. Since $\{g(1),\ldots,g(D)\}$ is also linearly independent over $\Z$, it follows that there is only one way to write $K$ as a $\Z$-linear combination of reals in $\{g(1),\ldots,g(D)\}$. Thus, for each $i\in\{1,\ldots,D\}$, there is a constant $c_i\in\Z$ such that $\mult{D\left(\overline{\{v\}}\right)}{i}=c_i$ for all $v\in V(G)$.
Now, since for all $A\subseteq V(G)$ and all $i\in\{1,\ldots,D\}$ we have that
\[\mult{D(A)}{i}+\mult{D(\bar{A})}{i}+\mult{D(A,\bar{A})}{i}=\mult{D(V(G)}{i},\]
we see that for $v\in V(G)$ and $i\in\{1,\ldots,D\}$, the quantity
\[\mult{D\left(\{v\},\overline{\{v\}}\right)}{i}=\mult{D(V(G))}{i}-c_i,\]
does not depend on $v$. Therefore, for all $u,v\in V(G)$ we have $\vec{D}(u)=\vec{D}(v)$, so $G$ is distance degree regular.
\end{proof}

\section{Maximizers of the Wiener index on paths and cycles}\label{section_wiener}

Let us establish a simple formula for the maximum value of the Wiener index of a set of vertices in a path $P_n$ with $n$ vertices.

\begin{lemma}\label{lemma_wiener_index_formula}
Suppose $A$ is a set of vertices in $P_n$ where  $2\leq |A|=m\leq n$. We enumerate $A$ in increasing order as $A=\{a_1,\ldots,a_m\}$. Then
\[W(A)=\sum_{i=1}^{m-1}i(m-i)(a_{i+1}-a_i).\]
\end{lemma}

\begin{proof}
We have, 
\[W(A)=\sum_{i=1}^{m-1}n_i(a_{i+1}-a_i),\]
where $n_i$ is the number of intervals of the form $\{a_j,\ldots,a_k\}$ such that $j<k$ and $\{a_i,\ldots,a_{i+1}\}\subseteq \{a_j,\ldots,a_k\}$. But an interval of the form $\{a_j,\ldots,a_k\}$ will contain $\{a_i,\ldots,a_{i+1}\}$ if and only if $j\in\{1,\ldots,i\}$ and $k\in\{i+1,\ldots,m\}$. Hence $n_i=i(m-i)$.
\end{proof}

The following proposition provides a characterization of the maximizers and local maximizers of the Wiener index $W$ on $P_n$.

\begin{proposition}\label{proposition_wiener_path}
Suppose $A$ is a set of vertices in $P_n$ where  $2\leq |A|=m\leq n$. The following are equivalent.
\begin{enumerate}
\item\label{item_prop_paths_max} $A$ is a maximizer of $W$ on $P_n$.
\item\label{item_prop_paths_loc_max} $A$ is a local maximizer of $W$ on $P_n$.
\item\label{item_prop_paths_cond} If $m$ is even then
\[A=\left\{1,\ldots,\frac{m}{2}\right\}\cup\left\{n-\frac{m}{2}+1,\ldots,n\right\},\] 
and if $m$ is odd then 
\[A=\left\{1,\ldots,\frac{m-1}{2}\right\}\cup\{j\}\cup\left\{n-\frac{m-1}{2}+1,\ldots,n\right\}\] for some integer $j$ with $\frac{m-1}{2}<j<n-\frac{m-1}{2}+1$.
\end{enumerate}
\end{proposition}

\begin{proof}
Throughout the proof we assume $A=\{a_1,\ldots,a_{m}\}$, where the elements of $A$ are enumerated in increasing order. We will prove that (\ref{item_prop_paths_max}) $\implies$ (\ref{item_prop_paths_loc_max}) $\implies$ (\ref{item_prop_paths_cond}) $\implies$ (\ref{item_prop_paths_max}). Clearly (\ref{item_prop_paths_max}) $\implies$ (\ref{item_prop_paths_loc_max}).

Let us prove (\ref{item_prop_paths_loc_max}) $\implies$ (\ref{item_prop_paths_cond}). Suppose (\ref{item_prop_paths_cond}) is false and $m$ is even. Then for some $i\in\{1,\ldots,m\}\setminus\{\frac{m}{2}\}$ we have $d(a_i,a_{i+1})>1$. Suppose $i<\frac{m}{2}$. Since $m$ is even we have $m\geq 2i+2$. Let $A'=(A\setminus\{a_{i+1}\})\cup\{a_{i+1}-1\}$ and notice that 
\[W(A')=W(A)-i+(m-i-1)=W(A)+m-2i-1\geq W(A)+1>W(A).\] 
Thus $A$ is not a local maximizer of $W$. The case in which $m$ is even and $i>\frac{m}{2}$ is similar. 

Now suppose (\ref{item_prop_paths_cond}) is false and $m$ is odd. Then for some $i\in\{1,\ldots,m-1\}\setminus\{\frac{m-1}{2},\frac{m+1}{2}\}$ we have $d(a_i,a_{i+1})>1$. Suppose $i<\frac{m-1}{2}$, then $m>2i+1$ and since $m$ is odd we have $m\geq 2i+3$. Let $A'=(A\setminus\{a_{i+1}\})\cup\{a_{i+1}-1\}$ and notice that
\[W(A')=W(A)-i+(m-i-1)=W(A)+m-2i-1\geq W(A)+2>W(A).\]
Thus $A$ is not a local maximizer of $W$. The case in which $m$ is odd and $i>\frac{m+1}{2}$ is similar. This establishes (\ref{item_prop_paths_loc_max}) $\implies$ (\ref{item_prop_paths_cond}).

Now we show that (\ref{item_prop_paths_cond}) $\implies$ (\ref{item_prop_paths_max}). Suppose $m$ is even. We must show that $A=\left\{1,\ldots,\frac{m}{2}\right\}\cup\left\{n-\frac{m}{2}+1,\ldots,n\right\}$ is a maximizer of $W$. Let $B\subseteq V(P_n)$ with $|B|=m$ be a maximizer of $W$, since we already proved that (\ref{item_prop_paths_max}) $\implies$ (\ref{item_prop_paths_cond}), it follows that $B=\left\{1,\ldots,\frac{m}{2}\right\}\cup\left\{n-\frac{m}{2}+1,\ldots,n\right\}$. But then $A=B$ and hence $A$ is a maximizer of $W$. The case in which $m$ is odd is similar.
\end{proof}

\begin{corollary}\label{corollary_ascent_on_paths}
Suppose $A$ is a set of vertices in $P_n$ where $2\leq |A|=\leq n$. Then the set $\ascent(P_n,W,A)$ is a maximizer of $W$.
\end{corollary}

Let us characterize the maximizers and local maximizers of $W$ on $C_n$ in terms of the chromatic and diatonic lengths studied by Clough and Douthett (see the beginning of Section \ref{section_me} for definitions). Let us note, that using the following theorem it is straight forward to classify all Wiener index maximizers on any given $n$-cycle; for example, up to rotations and reflections, all Wiener index maximizers on $C_7$ are depicted in Figure \ref{figure_WI_max_c7} and all Wiener index maximizers on $C_8$ are depicted in Figure \ref{figure_WI_max_c8}.

\begin{theorem}\label{theorem_w}
Suppose $A$ is a set of vertices in $C_n$ where  $2\leq|A|=m\leq n$. The following are equivalent.
\begin{enumerate}
\item\label{theorem_w_max} $A$ is a maximizer of $W$ on $C_n$.
\item\label{theorem_w_local} $A$ is a local maximizer of $W$ on $C_n$.
\item\label{theorem_w_top_spectra} If $m$ is odd then 
\[\supp(\spec_{\floor{\frac{m}{2}}}(A))\subseteq\left\{1,\ldots,\floor{\frac{n}{2}}\right\}\] and if $m$ is even then \[\supp(\spec_{\frac{m}{2}}(A))\subseteq\left\{\floor{\frac{n}{2}},\ceil{\frac{n}{2}}\right\}.\]
\item\label{theorem_w_all_spec} For all integers $k$ with $1\leq k< \frac{m}{2}$ we have $\sigma_k(A)=\sigma^*_k(A)$ and, in case $m$ is even, we have \[\supp(\sigma_{\frac{m}{2}}(A))=\left\{\floor{\frac{n}{2}}\right\}.\]
\end{enumerate}
\end{theorem}

\begin{proof}
We will show that (\ref{theorem_w_max}) $\implies$ (\ref{theorem_w_local}) $\implies$ (\ref{theorem_w_top_spectra}) $\implies$ (\ref{theorem_w_all_spec}) $\iff$ (\ref{theorem_w_max}). Notice that (\ref{theorem_w_max}) $\implies$ (\ref{theorem_w_local}) is trivial. Without loss of generality, let us assume that $3\leq |A|=m\leq n-2$.

Throughout the proof, we assume $A=\{a_0,\ldots,a_{m-1}\}$ where the elements of $A$ are enumerated in increasing order. Additionally, we also assume that indices of elements of $A$ are always reduced modulo $m$.

Let us show that (\ref{theorem_w_local}) $\implies$ (\ref{theorem_w_top_spectra}). Suppose (\ref{theorem_w_top_spectra}) is false. Suppose $m$ is odd and 
\[\supp(\spec_t(A))\not\subseteq\left\{1,\ldots,\floor{\frac{n}{2}}\right\}\] where $t=\floor{\frac{m}{2}}$. Then, there is some pair $(a_i,a_j)\in A\times A$ such that $\spn_A(a_i,a_j)=t$ and $d^*_{C_n}(a_i,a_j)>\diam(C_n)=\floor{\frac{n}{2}}$. Since $|A|=m\leq n-2$, we may let $k$ be the least nonnegative integer $\ell$ such that 
\[d^*_{C_n}(a_{i+\ell},a_{i+\ell+1})>1.\] As illustrated in Figure \ref{figure_a_pair_with_large_distance}, it follows that 
\begin{align*}
\spn_A(a_{i+k},a_{i+k+t})&=t,\\
d^*_{C_n}(a_{i+k},a_{i+k+t})&>\floor{\frac{n}{2}}
\end{align*}
and 
\[d^*_{C_n}(a_{i+k},a_{i+k+1})>1.\]
\begin{figure}
\begin{tikzpicture}[scale=0.7]
  % Circle
  \draw (0,0) circle (1.5);

  \coordinate (ai) at ({90}:1.5);
  \node[draw, label={[label distance=3pt]90:$a_i$}, circle, fill=black, inner sep=2pt] at (ai) {};

  \coordinate (aip1) at ({70}:1.5);
  \node[draw, circle, fill=black, inner sep=2pt] at (aip1) {};

  \draw [loosely dotted, very thick, domain=45:78] plot ({2*cos(\x)}, {2*sin(\x)});

  \coordinate (aip2) at ({50}:1.5);
  \node[draw, circle, fill=black, inner sep=2pt] at (aip2) {};

 \coordinate (aipk) at ({30}:1.5);
  \node[draw, label={[label distance=1pt]30:$a_{i+k}$}, circle, fill=black, inner sep=2pt] at (aipk) {};

\draw[thick] (10:1.4) -- (10:1.6);

 \coordinate (aipkpp1) at ({10}:1.5);
  \node[label={[label distance=3pt]0:$a_{i+k}+1$}] at (aipkpp1) {};

 \coordinate (aipkp1) at ({-10}:1.5);
  \node[draw, label={[label distance=1pt]-10:$a_{i+k+1}$}, circle, fill=black, inner sep=2pt] at (aipkp1) {};

  \coordinate (aj) at ({250}:1.5);
  \node[draw, label={[label distance=5pt]273:$a_j=a_{i+t}$}, circle, fill=black, inner sep=2pt] at (aj) {};

  \coordinate (ajp1) at ({230}:1.5);
  \node[draw, circle, fill=black, inner sep=2pt] at (ajp1) {};

\draw[thick] (210:1.4) -- (210:1.6);

  \coordinate (ajp2) at ({190}:1.5);
  \node[draw, circle, fill=black, inner sep=2pt] at (ajp2) {};

    \coordinate (ajp3) at ({170}:1.5);
  \node[draw, label={[label distance=5pt]170:$a_{i+k+t}$}, circle, fill=black, inner sep=2pt] at (ajp3) {};

   \draw [loosely dotted, very thick, domain=180:240] plot ({2*cos(\x)}, {2*sin(\x)});

  % Draw edges in the background layer
  \begin{pgfonlayer}{background}

  \end{pgfonlayer}

\end{tikzpicture}
\caption{\small \small A pair $(a_i,a_j)$ with $\spn_A(a_i,a_j)=t=\floor{\frac{m}{2}}$ and $d^*(a_i,a_j)>\floor{\frac{n}{2}}$.}\label{figure_a_pair_with_large_distance}
\end{figure}
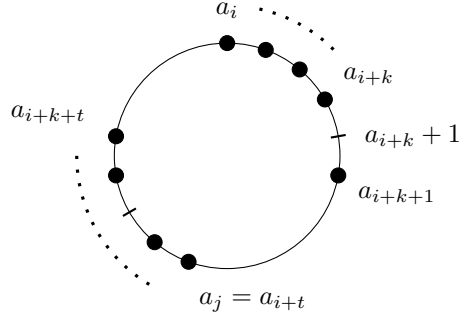Thus $a_{i+k}+1\notin A$. Let $A'=(A\setminus\{a_{i+k}\})\cup\{a_{i+k}+1\}$. Suppose $n$ is even. Since $\spn_A(a_{i+k},a_{i+k+t})=t$, the multiset $D(A')$ can be obtained from $D(A)$ by subtracting $1$ from at most $t-1$ members of $D(A)$ and adding $1$ to at least $t+1$ members of $D(A)$. Hence
\[W(A')-W(A)\geq t+1-(t-1)=2,\]
and thus $A$ is not a local maximizer of $W$ on $C_n$. Suppose $n$ is odd. Since the multiset $D(A')$ can be obtained from $D(A)$ by subtracting $1$ from at most $t-1$ members of $D(A)$ and adding $1$ to at least $t$ members of $D(A)$, we have
\[W(A')-W(A)\geq t-(t-1)=1,\]
and thus $A$ is not a local maximizer of $W$ on $C_n$.

The case in which $m$ is even is similar.

Let us show that (\ref{theorem_w_top_spectra}) $\implies$ (\ref{theorem_w_all_spec}). Let $t=\floor{\frac{m}{2}}$ and suppose that for all $d\in\sigma_t^*(A)$ we have $d\leq\floor{\frac{n}{2}}$ when $m$ is odd and $d\in\left\{\floor{\frac{n}{2}},\ceil{\frac{n}{2}}\right\}$ when $m$ is even. We must show that for all integers $k$ with $1\leq k< \frac{m}{2}$ we have $\sigma_k(A)=\sigma^*_k(A)$ and, in case $m$ is even, we have $\supp(\sigma_{\frac{m}{2}}(A))=\left\{\floor{\frac{n}{2}}\right\}$. 

Fix an integer $k$ with $1\leq k<\frac{m}{2}$. Suppose $(a_i,a_j)$ is an ordered pair of elements of $A$ with $\spn_A(a_i,a_j)=k$, that is, $k$ is the least positive integer such that $k=j-1\mod n$. Then, assuming indicies have been reduced modulo $m$, 
\[\spn_A(a_i,a_{j+t-k})=t,\]
and hence
\[\d(a_i,a_j)\leq\d(a_i,a_{j+t-k}).\]
If $m$ is odd then $\d(a_i,a_{j+t-k})\leq\floor{\frac{n}{2}}$ and thus $d(a_i,a_j)=\d(a_i,a_j)$. If $m$ is even then 
\[\d(a_i,a_{j+t-k})\in\left\{\floor{\frac{n}{2}},\ceil{\frac{n}{2}}\right\},\]
and since $k<t=\frac{m}{2}$, we see that $\d(a_i,a_j)<\d(a_i,a_{j+t-k})$ and hence $d(a_i,a_j)=d^*(a_i,a_j)$. This implies $\sigma_k(A)=\sigma_k^*(A)$.

Now, suppose $m$ is even. Then, by assumption $t=\frac{m}{2}$ and for all $d\in\sigma^*_{\frac{m}{2}}(A)$ we have $d\in\left\{\floor{\frac{n}{2}},\ceil{\frac{n}{2}}\right\}$, which clearly implies $\supp(\sigma_{\frac{m}{2}}(A))=\left\{\floor{\frac{n}{2}}\right\}$.

\begin{figure}
\begin{subfigure}[b]{0.08\textwidth}
    \centering
    \begin{tikzpicture}[thick, scale=0.4]
    \tikzstyle{vertex}=[circle, draw, fill=white, inner sep=0mm, minimum size=1.5mm, line width=0.5pt]
    \tikzstyle{edge}=[thin]
    \node[vertex, fill=white] (0) at (90:1) {};\node[vertex, fill=white] (1) at (270/7:1) {};\node[vertex, fill=white] (2) at (-90/7:1) {};\node[vertex, fill=white] (3) at (-450/7:1) {};\node[vertex, fill=white] (4) at (-810/7:1) {};\node[vertex, fill=white] (5) at (-1170/7:1) {};\node[vertex, fill=white] (6) at (-1530/7:1) {};\draw[edge] (0) -- (1);\draw[edge] (1) -- (2);\draw[edge] (2) -- (3);\draw[edge] (3) -- (4);\draw[edge] (4) -- (5);\draw[edge] (5) -- (6);\draw[edge] (6) -- (0);
    \end{tikzpicture}
\end{subfigure}
\begin{subfigure}[b]{0.08\textwidth}
    \centering
    \begin{tikzpicture}[thick, scale=0.4]
    \tikzstyle{vertex}=[circle, draw, fill=white, inner sep=0mm, minimum size=1.5mm, line width=0.5pt]
    \tikzstyle{edge}=[thin]
    \node[vertex, fill=black] (0) at (90:1) {};\node[vertex, fill=white] (1) at (270/7:1) {};\node[vertex, fill=white] (2) at (-90/7:1) {};\node[vertex, fill=white] (3) at (-450/7:1) {};\node[vertex, fill=white] (4) at (-810/7:1) {};\node[vertex, fill=white] (5) at (-1170/7:1) {};\node[vertex, fill=white] (6) at (-1530/7:1) {};\draw[edge] (0) -- (1);\draw[edge] (1) -- (2);\draw[edge] (2) -- (3);\draw[edge] (3) -- (4);\draw[edge] (4) -- (5);\draw[edge] (5) -- (6);\draw[edge] (6) -- (0);
    \end{tikzpicture}
\end{subfigure}
\begin{subfigure}[b]{0.08\textwidth}
    \centering
    \begin{tikzpicture}[thick, scale=0.4]
    \tikzstyle{vertex}=[circle, draw, fill=white, inner sep=0mm, minimum size=1.5mm, line width=0.5pt]
    \tikzstyle{edge}=[thin]
    \node[vertex, fill=black] (0) at (90:1) {};\node[vertex, fill=white] (1) at (270/7:1) {};\node[vertex, fill=white] (2) at (-90/7:1) {};\node[vertex, fill=black] (3) at (-450/7:1) {};\node[vertex, fill=white] (4) at (-810/7:1) {};\node[vertex, fill=white] (5) at (-1170/7:1) {};\node[vertex, fill=white] (6) at (-1530/7:1) {};\draw[edge] (0) -- (1);\draw[edge] (1) -- (2);\draw[edge] (2) -- (3);\draw[edge] (3) -- (4);\draw[edge] (4) -- (5);\draw[edge] (5) -- (6);\draw[edge] (6) -- (0);
    \end{tikzpicture}
\end{subfigure}
\begin{subfigure}[b]{0.08\textwidth}
    \centering
    \begin{tikzpicture}[thick, scale=0.4]
    \tikzstyle{vertex}=[circle, draw, fill=white, inner sep=0mm, minimum size=1.5mm, line width=0.5pt]
    \tikzstyle{edge}=[thin]
    \node[vertex, fill=black] (0) at (90:1) {};\node[vertex, fill=black] (1) at (270/7:1) {};\node[vertex, fill=white] (2) at (-90/7:1) {};\node[vertex, fill=white] (3) at (-450/7:1) {};\node[vertex, fill=black] (4) at (-810/7:1) {};\node[vertex, fill=white] (5) at (-1170/7:1) {};\node[vertex, fill=white] (6) at (-1530/7:1) {};\draw[edge] (0) -- (1);\draw[edge] (1) -- (2);\draw[edge] (2) -- (3);\draw[edge] (3) -- (4);\draw[edge] (4) -- (5);\draw[edge] (5) -- (6);\draw[edge] (6) -- (0);
    \end{tikzpicture}
\end{subfigure}
\begin{subfigure}[b]{0.08\textwidth}
    \centering
    \begin{tikzpicture}[thick, scale=0.4]
    \tikzstyle{vertex}=[circle, draw, fill=white, inner sep=0mm, minimum size=1.5mm, line width=0.5pt]
    \tikzstyle{edge}=[thin]
    \node[vertex, fill=black] (0) at (90:1) {};\node[vertex, fill=white] (1) at (270/7:1) {};\node[vertex, fill=black] (2) at (-90/7:1) {};\node[vertex, fill=white] (3) at (-450/7:1) {};\node[vertex, fill=black] (4) at (-810/7:1) {};\node[vertex, fill=white] (5) at (-1170/7:1) {};\node[vertex, fill=white] (6) at (-1530/7:1) {};\draw[edge] (0) -- (1);\draw[edge] (1) -- (2);\draw[edge] (2) -- (3);\draw[edge] (3) -- (4);\draw[edge] (4) -- (5);\draw[edge] (5) -- (6);\draw[edge] (6) -- (0);
    \end{tikzpicture}
\end{subfigure}
\begin{subfigure}[b]{0.08\textwidth}
    \centering
    \begin{tikzpicture}[thick, scale=0.4]
    \tikzstyle{vertex}=[circle, draw, fill=white, inner sep=0mm, minimum size=1.5mm, line width=0.5pt]
    \tikzstyle{edge}=[thin]
    \node[vertex, fill=black] (0) at (90:1) {};\node[vertex, fill=black] (1) at (270/7:1) {};\node[vertex, fill=white] (2) at (-90/7:1) {};\node[vertex, fill=black] (3) at (-450/7:1) {};\node[vertex, fill=black] (4) at (-810/7:1) {};\node[vertex, fill=white] (5) at (-1170/7:1) {};\node[vertex, fill=white] (6) at (-1530/7:1) {};\draw[edge] (0) -- (1);\draw[edge] (1) -- (2);\draw[edge] (2) -- (3);\draw[edge] (3) -- (4);\draw[edge] (4) -- (5);\draw[edge] (5) -- (6);\draw[edge] (6) -- (0);
    \end{tikzpicture}
\end{subfigure}
\begin{subfigure}[b]{0.08\textwidth}
    \centering
    \begin{tikzpicture}[thick, scale=0.4]
    \tikzstyle{vertex}=[circle, draw, fill=white, inner sep=0mm, minimum size=1.5mm, line width=0.5pt]
    \tikzstyle{edge}=[thin]
    \node[vertex, fill=black] (0) at (90:1) {};\node[vertex, fill=black] (1) at (270/7:1) {};\node[vertex, fill=white] (2) at (-90/7:1) {};\node[vertex, fill=black] (3) at (-450/7:1) {};\node[vertex, fill=white] (4) at (-810/7:1) {};\node[vertex, fill=black] (5) at (-1170/7:1) {};\node[vertex, fill=white] (6) at (-1530/7:1) {};\draw[edge] (0) -- (1);\draw[edge] (1) -- (2);\draw[edge] (2) -- (3);\draw[edge] (3) -- (4);\draw[edge] (4) -- (5);\draw[edge] (5) -- (6);\draw[edge] (6) -- (0);
    \end{tikzpicture}
\end{subfigure}
\begin{subfigure}[b]{0.08\textwidth}
    \centering
    \begin{tikzpicture}[thick, scale=0.4]
    \tikzstyle{vertex}=[circle, draw, fill=white, inner sep=0mm, minimum size=1.5mm, line width=0.5pt]
    \tikzstyle{edge}=[thin]
    \node[vertex, fill=black] (0) at (90:1) {};\node[vertex, fill=black] (1) at (270/7:1) {};\node[vertex, fill=black] (2) at (-90/7:1) {};\node[vertex, fill=white] (3) at (-450/7:1) {};\node[vertex, fill=black] (4) at (-810/7:1) {};\node[vertex, fill=black] (5) at (-1170/7:1) {};\node[vertex, fill=white] (6) at (-1530/7:1) {};\draw[edge] (0) -- (1);\draw[edge] (1) -- (2);\draw[edge] (2) -- (3);\draw[edge] (3) -- (4);\draw[edge] (4) -- (5);\draw[edge] (5) -- (6);\draw[edge] (6) -- (0);
    \end{tikzpicture}
\end{subfigure}
\begin{subfigure}[b]{0.08\textwidth}
    \centering
    \begin{tikzpicture}[thick, scale=0.4]
    \tikzstyle{vertex}=[circle, draw, fill=white, inner sep=0mm, minimum size=1.5mm, line width=0.5pt]
    \tikzstyle{edge}=[thin]
    \node[vertex, fill=black] (0) at (90:1) {};\node[vertex, fill=black] (1) at (270/7:1) {};\node[vertex, fill=black] (2) at (-90/7:1) {};\node[vertex, fill=black] (3) at (-450/7:1) {};\node[vertex, fill=black] (4) at (-810/7:1) {};\node[vertex, fill=black] (5) at (-1170/7:1) {};\node[vertex, fill=white] (6) at (-1530/7:1) {};\draw[edge] (0) -- (1);\draw[edge] (1) -- (2);\draw[edge] (2) -- (3);\draw[edge] (3) -- (4);\draw[edge] (4) -- (5);\draw[edge] (5) -- (6);\draw[edge] (6) -- (0);
    \end{tikzpicture}
\end{subfigure}
\begin{subfigure}[b]{0.08\textwidth}\centering\begin{tikzpicture}[thick, scale=0.4]\tikzstyle{vertex}=[circle, draw, fill=white, inner sep=0mm, minimum size=1.5mm, line width=0.5pt]\tikzstyle{edge}=[thin]\node[vertex, fill=black] (0) at (90:1) {};\node[vertex, fill=black] (1) at (270/7:1) {};\node[vertex, fill=black] (2) at (-90/7:1) {};\node[vertex, fill=black] (3) at (-450/7:1) {};\node[vertex, fill=black] (4) at (-810/7:1) {};\node[vertex, fill=black] (5) at (-1170/7:1) {};\node[vertex, fill=black] (6) at (-1530/7:1) {};\draw[edge] (0) -- (1);\draw[edge] (1) -- (2);\draw[edge] (2) -- (3);\draw[edge] (3) -- (4);\draw[edge] (4) -- (5);\draw[edge] (5) -- (6);\draw[edge] (6) -- (0);\end{tikzpicture}\end{subfigure}
        \caption{\small All Wiener index maximizers on $C_7$ up to rotations and reflections.}
        \label{figure_WI_max_c7}
    \end{figure}
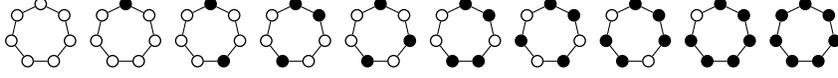

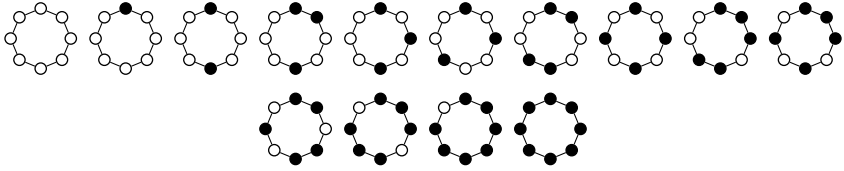
\begin{figure}
\centering
\begin{subfigure}[b]{0.08\textwidth}
    \centering
    \begin{tikzpicture}[thick, scale=0.4]
    \tikzstyle{vertex}=[circle, draw, fill=white, inner sep=0mm, minimum size=1.5mm, line width=0.5pt]
    \tikzstyle{edge}=[thin]
    \node[vertex, fill=white] (0) at (90:1) {};\node[vertex, fill=white] (1) at (45:1) {};\node[vertex, fill=white] (2) at (0:1) {};\node[vertex, fill=white] (3) at (-45:1) {};\node[vertex, fill=white] (4) at (-90:1) {};\node[vertex, fill=white] (5) at (-135:1) {};\node[vertex, fill=white] (6) at (-180:1) {};\node[vertex, fill=white] (7) at (-225:1) {};\draw[edge] (0) -- (1);\draw[edge] (1) -- (2);\draw[edge] (2) -- (3);\draw[edge] (3) -- (4);\draw[edge] (4) -- (5);\draw[edge] (5) -- (6);\draw[edge] (6) -- (7);\draw[edge] (7) -- (0);
    \end{tikzpicture}
\end{subfigure}
\begin{subfigure}[b]{0.08\textwidth}
    \centering
    \begin{tikzpicture}[thick, scale=0.4]
    \tikzstyle{vertex}=[circle, draw, fill=white, inner sep=0mm, minimum size=1.5mm, line width=0.5pt]
    \tikzstyle{edge}=[thin]
    \node[vertex, fill=black] (0) at (90:1) {};\node[vertex, fill=white] (1) at (45:1) {};\node[vertex, fill=white] (2) at (0:1) {};\node[vertex, fill=white] (3) at (-45:1) {};\node[vertex, fill=white] (4) at (-90:1) {};\node[vertex, fill=white] (5) at (-135:1) {};\node[vertex, fill=white] (6) at (-180:1) {};\node[vertex, fill=white] (7) at (-225:1) {};\draw[edge] (0) -- (1);\draw[edge] (1) -- (2);\draw[edge] (2) -- (3);\draw[edge] (3) -- (4);\draw[edge] (4) -- (5);\draw[edge] (5) -- (6);\draw[edge] (6) -- (7);\draw[edge] (7) -- (0);
    \end{tikzpicture}
\end{subfigure}
\begin{subfigure}[b]{0.08\textwidth}
    \centering
    \begin{tikzpicture}[thick, scale=0.4]
    \tikzstyle{vertex}=[circle, draw, fill=white, inner sep=0mm, minimum size=1.5mm, line width=0.5pt]
    \tikzstyle{edge}=[thin]
    \node[vertex, fill=black] (0) at (90:1) {};\node[vertex, fill=white] (1) at (45:1) {};\node[vertex, fill=white] (2) at (0:1) {};\node[vertex, fill=white] (3) at (-45:1) {};\node[vertex, fill=black] (4) at (-90:1) {};\node[vertex, fill=white] (5) at (-135:1) {};\node[vertex, fill=white] (6) at (-180:1) {};\node[vertex, fill=white] (7) at (-225:1) {};\draw[edge] (0) -- (1);\draw[edge] (1) -- (2);\draw[edge] (2) -- (3);\draw[edge] (3) -- (4);\draw[edge] (4) -- (5);\draw[edge] (5) -- (6);\draw[edge] (6) -- (7);\draw[edge] (7) -- (0);
    \end{tikzpicture}
\end{subfigure}
\begin{subfigure}[b]{0.08\textwidth}
    \centering
    \begin{tikzpicture}[thick, scale=0.4]
    \tikzstyle{vertex}=[circle, draw, fill=white, inner sep=0mm, minimum size=1.5mm, line width=0.5pt]
    \tikzstyle{edge}=[thin]
    \node[vertex, fill=black] (0) at (90:1) {};\node[vertex, fill=black] (1) at (45:1) {};\node[vertex, fill=white] (2) at (0:1) {};\node[vertex, fill=white] (3) at (-45:1) {};\node[vertex, fill=black] (4) at (-90:1) {};\node[vertex, fill=white] (5) at (-135:1) {};\node[vertex, fill=white] (6) at (-180:1) {};\node[vertex, fill=white] (7) at (-225:1) {};\draw[edge] (0) -- (1);\draw[edge] (1) -- (2);\draw[edge] (2) -- (3);\draw[edge] (3) -- (4);\draw[edge] (4) -- (5);\draw[edge] (5) -- (6);\draw[edge] (6) -- (7);\draw[edge] (7) -- (0);
    \end{tikzpicture}
\end{subfigure}
\begin{subfigure}[b]{0.08\textwidth}
    \centering
    \begin{tikzpicture}[thick, scale=0.4]
    \tikzstyle{vertex}=[circle, draw, fill=white, inner sep=0mm, minimum size=1.5mm, line width=0.5pt]
    \tikzstyle{edge}=[thin]
    \node[vertex, fill=black] (0) at (90:1) {};\node[vertex, fill=white] (1) at (45:1) {};\node[vertex, fill=black] (2) at (0:1) {};\node[vertex, fill=white] (3) at (-45:1) {};\node[vertex, fill=black] (4) at (-90:1) {};\node[vertex, fill=white] (5) at (-135:1) {};\node[vertex, fill=white] (6) at (-180:1) {};\node[vertex, fill=white] (7) at (-225:1) {};\draw[edge] (0) -- (1);\draw[edge] (1) -- (2);\draw[edge] (2) -- (3);\draw[edge] (3) -- (4);\draw[edge] (4) -- (5);\draw[edge] (5) -- (6);\draw[edge] (6) -- (7);\draw[edge] (7) -- (0);
    \end{tikzpicture}
\end{subfigure}
\begin{subfigure}[b]{0.08\textwidth}
    \centering
    \begin{tikzpicture}[thick, scale=0.4]
    \tikzstyle{vertex}=[circle, draw, fill=white, inner sep=0mm, minimum size=1.5mm, line width=0.5pt]
    \tikzstyle{edge}=[thin]
    \node[vertex, fill=black] (0) at (90:1) {};\node[vertex, fill=white] (1) at (45:1) {};\node[vertex, fill=black] (2) at (0:1) {};\node[vertex, fill=white] (3) at (-45:1) {};\node[vertex, fill=white] (4) at (-90:1) {};\node[vertex, fill=black] (5) at (-135:1) {};\node[vertex, fill=white] (6) at (-180:1) {};\node[vertex, fill=white] (7) at (-225:1) {};\draw[edge] (0) -- (1);\draw[edge] (1) -- (2);\draw[edge] (2) -- (3);\draw[edge] (3) -- (4);\draw[edge] (4) -- (5);\draw[edge] (5) -- (6);\draw[edge] (6) -- (7);\draw[edge] (7) -- (0);
    \end{tikzpicture}
\end{subfigure}
\begin{subfigure}[b]{0.08\textwidth}
    \centering
    \begin{tikzpicture}[thick, scale=0.4]
    \tikzstyle{vertex}=[circle, draw, fill=white, inner sep=0mm, minimum size=1.5mm, line width=0.5pt]
    \tikzstyle{edge}=[thin]
    \node[vertex, fill=black] (0) at (90:1) {};\node[vertex, fill=black] (1) at (45:1) {};\node[vertex, fill=white] (2) at (0:1) {};\node[vertex, fill=white] (3) at (-45:1) {};\node[vertex, fill=black] (4) at (-90:1) {};\node[vertex, fill=black] (5) at (-135:1) {};\node[vertex, fill=white] (6) at (-180:1) {};\node[vertex, fill=white] (7) at (-225:1) {};\draw[edge] (0) -- (1);\draw[edge] (1) -- (2);\draw[edge] (2) -- (3);\draw[edge] (3) -- (4);\draw[edge] (4) -- (5);\draw[edge] (5) -- (6);\draw[edge] (6) -- (7);\draw[edge] (7) -- (0);
    \end{tikzpicture}
\end{subfigure}
\begin{subfigure}[b]{0.08\textwidth}
    \centering
    \begin{tikzpicture}[thick, scale=0.4]
    \tikzstyle{vertex}=[circle, draw, fill=white, inner sep=0mm, minimum size=1.5mm, line width=0.5pt]
    \tikzstyle{edge}=[thin]
    \node[vertex, fill=black] (0) at (90:1) {};\node[vertex, fill=white] (1) at (45:1) {};\node[vertex, fill=black] (2) at (0:1) {};\node[vertex, fill=white] (3) at (-45:1) {};\node[vertex, fill=black] (4) at (-90:1) {};\node[vertex, fill=white] (5) at (-135:1) {};\node[vertex, fill=black] (6) at (-180:1) {};\node[vertex, fill=white] (7) at (-225:1) {};\draw[edge] (0) -- (1);\draw[edge] (1) -- (2);\draw[edge] (2) -- (3);\draw[edge] (3) -- (4);\draw[edge] (4) -- (5);\draw[edge] (5) -- (6);\draw[edge] (6) -- (7);\draw[edge] (7) -- (0);
    \end{tikzpicture}
\end{subfigure}
\begin{subfigure}[b]{0.08\textwidth}
    \centering
    \begin{tikzpicture}[thick, scale=0.4]
    \tikzstyle{vertex}=[circle, draw, fill=white, inner sep=0mm, minimum size=1.5mm, line width=0.5pt]
    \tikzstyle{edge}=[thin]
    \node[vertex, fill=black] (0) at (90:1) {};\node[vertex, fill=black] (1) at (45:1) {};\node[vertex, fill=black] (2) at (0:1) {};\node[vertex, fill=white] (3) at (-45:1) {};\node[vertex, fill=black] (4) at (-90:1) {};\node[vertex, fill=black] (5) at (-135:1) {};\node[vertex, fill=white] (6) at (-180:1) {};\node[vertex, fill=white] (7) at (-225:1) {};\draw[edge] (0) -- (1);\draw[edge] (1) -- (2);\draw[edge] (2) -- (3);\draw[edge] (3) -- (4);\draw[edge] (4) -- (5);\draw[edge] (5) -- (6);\draw[edge] (6) -- (7);\draw[edge] (7) -- (0);
    \end{tikzpicture}
\end{subfigure}
\begin{subfigure}[b]{0.08\textwidth}
    \centering
    \begin{tikzpicture}[thick, scale=0.4]
    \tikzstyle{vertex}=[circle, draw, fill=white, inner sep=0mm, minimum size=1.5mm, line width=0.5pt]
    \tikzstyle{edge}=[thin]
    \node[vertex, fill=black] (0) at (90:1) {};\node[vertex, fill=black] (1) at (45:1) {};\node[vertex, fill=black] (2) at (0:1) {};\node[vertex, fill=white] (3) at (-45:1) {};\node[vertex, fill=black] (4) at (-90:1) {};\node[vertex, fill=white] (5) at (-135:1) {};\node[vertex, fill=black] (6) at (-180:1) {};\node[vertex, fill=white] (7) at (-225:1) {};\draw[edge] (0) -- (1);\draw[edge] (1) -- (2);\draw[edge] (2) -- (3);\draw[edge] (3) -- (4);\draw[edge] (4) -- (5);\draw[edge] (5) -- (6);\draw[edge] (6) -- (7);\draw[edge] (7) -- (0);
    \end{tikzpicture}
\end{subfigure}

\vspace{2mm}

\begin{subfigure}[b]{0.08\textwidth}
    \centering
    \begin{tikzpicture}[thick, scale=0.4]
    \tikzstyle{vertex}=[circle, draw, fill=white, inner sep=0mm, minimum size=1.5mm, line width=0.5pt]
    \tikzstyle{edge}=[thin]
    \node[vertex, fill=black] (0) at (90:1) {};\node[vertex, fill=black] (1) at (45:1) {};\node[vertex, fill=white] (2) at (0:1) {};\node[vertex, fill=black] (3) at (-45:1) {};\node[vertex, fill=black] (4) at (-90:1) {};\node[vertex, fill=white] (5) at (-135:1) {};\node[vertex, fill=black] (6) at (-180:1) {};\node[vertex, fill=white] (7) at (-225:1) {};\draw[edge] (0) -- (1);\draw[edge] (1) -- (2);\draw[edge] (2) -- (3);\draw[edge] (3) -- (4);\draw[edge] (4) -- (5);\draw[edge] (5) -- (6);\draw[edge] (6) -- (7);\draw[edge] (7) -- (0);
    \end{tikzpicture}
\end{subfigure}
\begin{subfigure}[b]{0.08\textwidth}
    \centering
    \begin{tikzpicture}[thick, scale=0.4]
    \tikzstyle{vertex}=[circle, draw, fill=white, inner sep=0mm, minimum size=1.5mm, line width=0.5pt]
    \tikzstyle{edge}=[thin]
    \node[vertex, fill=black] (0) at (90:1) {};\node[vertex, fill=black] (1) at (45:1) {};\node[vertex, fill=black] (2) at (0:1) {};\node[vertex, fill=white] (3) at (-45:1) {};\node[vertex, fill=black] (4) at (-90:1) {};\node[vertex, fill=black] (5) at (-135:1) {};\node[vertex, fill=black] (6) at (-180:1) {};\node[vertex, fill=white] (7) at (-225:1) {};\draw[edge] (0) -- (1);\draw[edge] (1) -- (2);\draw[edge] (2) -- (3);\draw[edge] (3) -- (4);\draw[edge] (4) -- (5);\draw[edge] (5) -- (6);\draw[edge] (6) -- (7);\draw[edge] (7) -- (0);
    \end{tikzpicture}
\end{subfigure}
\begin{subfigure}[b]{0.08\textwidth}
    \centering
    \begin{tikzpicture}[thick, scale=0.4]
    \tikzstyle{vertex}=[circle, draw, fill=white, inner sep=0mm, minimum size=1.5mm, line width=0.5pt]
    \tikzstyle{edge}=[thin]
    \node[vertex, fill=black] (0) at (90:1) {};\node[vertex, fill=black] (1) at (45:1) {};\node[vertex, fill=black] (2) at (0:1) {};\node[vertex, fill=black] (3) at (-45:1) {};\node[vertex, fill=black] (4) at (-90:1) {};\node[vertex, fill=black] (5) at (-135:1) {};\node[vertex, fill=black] (6) at (-180:1) {};\node[vertex, fill=white] (7) at (-225:1) {};\draw[edge] (0) -- (1);\draw[edge] (1) -- (2);\draw[edge] (2) -- (3);\draw[edge] (3) -- (4);\draw[edge] (4) -- (5);\draw[edge] (5) -- (6);\draw[edge] (6) -- (7);\draw[edge] (7) -- (0);
    \end{tikzpicture}
\end{subfigure}
\begin{subfigure}[b]{0.08\textwidth}
    \centering
    \begin{tikzpicture}[thick, scale=0.4]
    \tikzstyle{vertex}=[circle, draw, fill=white, inner sep=0mm, minimum size=1.5mm, line width=0.5pt]
    \tikzstyle{edge}=[thin]
    \node[vertex, fill=black] (0) at (90:1) {};\node[vertex, fill=black] (1) at (45:1) {};\node[vertex, fill=black] (2) at (0:1) {};\node[vertex, fill=black] (3) at (-45:1) {};\node[vertex, fill=black] (4) at (-90:1) {};\node[vertex, fill=black] (5) at (-135:1) {};\node[vertex, fill=black] (6) at (-180:1) {};\node[vertex, fill=black] (7) at (-225:1) {};\draw[edge] (0) -- (1);\draw[edge] (1) -- (2);\draw[edge] (2) -- (3);\draw[edge] (3) -- (4);\draw[edge] (4) -- (5);\draw[edge] (5) -- (6);\draw[edge] (6) -- (7);\draw[edge] (7) -- (0);
    \end{tikzpicture}
\end{subfigure}
\caption{\small All Wiener index maximizers on $C_8$ up to rotations and reflections.}\label{figure_WI_max_c8}
\end{figure}

Next, let us prove that (\ref{theorem_w_all_spec}) $\implies$ (\ref{theorem_w_max}). Suppose that for all integers $k$ with $1\leq k< \frac{m}{2}$ we have $\sigma_k(A)=\sigma^*_k(A)$ and, in case $m$ is even, we have $\supp(\sigma_{\frac{m}{2}}(A))=\left\{\floor{\frac{n}{2}}\right\}$. Then $W(A)=W(J^0_{n,m})$. We must show that $A$ is a maximizer of $W$. Let $B\subseteq V(C_n)$ with $|B|=m$ be a maximizer of $W$. Since we already know (\ref{theorem_w_max}) $\implies$ (\ref{theorem_w_all_spec}), it follows that for each integer $k$ with $1\leq k<\frac{m}{2}$ we have $\sigma_k(B)=\sigma^*_k(B)$ and if $m$ is even then $\supp(\sigma_{\frac{m}{2}}(B))=\left\{\floor{\frac{n}{2}}\right\}$. Hence $W(B)=W(J^0_{n,m})=W(A)$ and therefore $A$ is a maximizer of $W$.
\end{proof}

\begin{corollary}\label{corollary_ascent_on_cycles}
Suppose $A$ is a set of vertices in $C_n$ where $2\leq A=m\leq n$. Then the set $\ascent(C_n,W,A)$ is a maximizer of $W$.
\end{corollary}

Using either Theorem \ref{theorem_w}(\ref{theorem_w_top_spectra}) or Theorem \ref{theorem_w}(\ref{theorem_w_all_spec}), it is straightforward to verify the following.

\begin{corollary} \label{corollary_relationship}
Suppose $A$ is a set of vertices in $C_n$. If $A$ is maximally even then $A$ is a maximizer of $W$. However, not every maximizer of $W$ on $C_n$ is maximally even.
\end{corollary}

Finally, let us prove a characterization of the maximizers and local maximizers of $W$ on $C_n$ using the balanced conditions in Definition \ref{definition_balanced}. We say that a set of vertices $X$ of a finite graph $G$ is \emph{connected} if the induced subgraph $G[X]$ is connected. A partition $\mathcal{P}$ of the vertices of a graph is called \emph{equitable} if for all blocks $P,Q\in\mathcal{P}$ the cardinalities $|P|$ and $|Q|$ differ by at most one. 

\begin{definition}\label{definition_balanced}
Suppose $G$ is a finite simple connected graph and $A$ is a set of vertices in $G$.
\begin{enumerate}
\item We say that $A$ is \emph{balanced in $G$} if for every equitable $2$-partition $\mathcal{P}$ of $V(G)$ with connected blocks, the cardinalities of $A$ intersected with the two blocks of $\P$ differ by at most one.
\item We say that $A$ is \emph{weakly balanced in $G$} if and only if for every equitable $2$-partition $\mathcal{P}$ of $V(G)$ with connected blocks, the cardinalities of $A$ intersected with the two blocks of $\P$ differ by at most two, and furthermore, whenever the cardinalities of $A$ intersected with the two blocks of such a $\P$ differ by exactly $2$, the block containing more points of $A$ must be the larger block, that is, $P,Q\in\mathcal{P}$ with $|A\cap Q|=|A\cap P|+2$ implies $|P|<|Q|$.
\end{enumerate}
\end{definition}

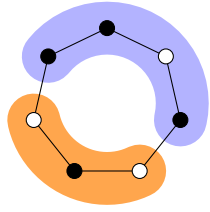
\begin{figure}
\begin{tikzpicture}[scale=0.9]
  % Define vertices with selective filling
  \foreach \i in {0,2,4,6} {
    \coordinate (v\i) at ({90 - (360/7)*\i}:1);
    \node[draw, circle, fill=black, text=white, inner sep=2pt] at (v\i) {};
  }
  
  % Define the rest of the vertices without filling
  \foreach \i in {1,3,5} {
    \coordinate (v\i) at ({90 - (360/7)*\i}:1);
    \node[draw, circle, fill=white, inner sep=2pt] at (v\i) {};
  }
  
  % Draw edges in the background layer
  \begin{pgfonlayer}{background}
  \draw [blue, line width=20pt, line cap=round, draw opacity=0.3, domain=90+(360/7):90-2*(360/7)] plot ({cos(\x)}, {sin(\x)});
  \draw [orange,line width=20pt, line cap=round, draw opacity=0.7,domain=90+2*(360/7):90+4*(360/7)] plot ({cos(\x)}, {sin(\x)});
    \foreach \i in {0,...,5} {
      \draw (v\i) -- (v\the\numexpr\i+1\relax);
    }
    \draw (v6) -- (v0); % Closing edge
  \end{pgfonlayer}

\end{tikzpicture}
\caption{\small \small A maximizer of the Wiener index on $C_7$ with cardinality $4$ that is weakly balanced but not balanced.}
\end{figure}

In the following theorem we charaterize the Wiener index maximizers on cycle graphs in terms of balanced and weakly balanced sets; let us note that it is instructive to verify the following characterizations in the case of $C_7$ (see Figure \ref{figure_WI_max_c7}) and $C_8$ (see Figure \ref{figure_WI_max_c8}).

\begin{theorem}\label{theorem_balanced}
Suppose $A$ is a set of vertices in $C_n$ where  $2\leq|A|=m\leq n$. 
\begin{enumerate}
\item\label{item_thm_bal_balanced} Suppose $n$ is even or $m$ is odd. Then $A$ is a maximizer of $W$ on $C_n$ if and only if $A$ is balanced.
\item\label{item_thm_bal_wbnb} Suppose $n$ is odd and $m$ is even. Then $A$ is a maximizer of $W$ on $C_n$ if and only if $A$ is weakly balanced. Furthermore, in this case, every maximizer of the Wiener index on $C_n$ with cardinality $m$ is not balanced.
\item\label{item_thm_bal_weakly_balanced} $A$ is a maximizer of $W$ if and only if it is weakly balanced.
\end{enumerate}
\end{theorem}

\begin{proof}

%So, in particular, when $U=V(C_n)$, we take $u_i=i$ for $i\in\{0,\ldots,n-1\}$, and then for $i,j\in\{0,\ldots,n-1\}$, $\cw{C_n}{u_i}{u_j}$ denotes the usual clockwise interval in $C_n$ from $u_i$ to $u_j$.

Fix a set $A=\{a_0,\ldots,a_{m-1}\}$ of vertices of $C_n$, where the elements of $A$ are enumerated in increasing order. Since balanced implies weakly balanced, (\ref{item_thm_bal_weakly_balanced}) easily follows from (\ref{item_thm_bal_balanced}) and (\ref{item_thm_bal_wbnb}).

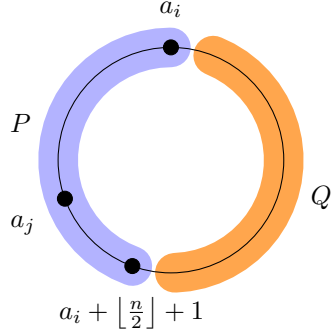
\begin{figure}
\begin{tikzpicture}[scale=0.7]
  % Circle
  \draw (0,0) circle (1.5);

  \coordinate (ai) at ({90}:1.5);
  \node[draw, label={[label distance=5pt]90:$a_i$}, circle, fill=black, inner sep=2pt] at (ai) {};

{[label distance=1cm]30:label}

  \coordinate (aj) at ({200}:1.5);
  \node[draw, label={[label distance=5pt]200:$a_j$}, circle, fill=black, inner sep=2pt] at (aj) {};

  \coordinate (aiplus) at ({250}:1.5);
  \node[draw, label={[label distance=5pt]268:$a_i+\floor{\frac{n}{2}}+1$}, circle, fill=black, inner sep=2pt] at (aiplus) {};

  % Draw edges in the background layer
  \begin{pgfonlayer}{background}
  \draw [blue,line width=10pt, line cap=round, draw opacity=0.3,domain=250:90] plot ({1.5*cos(\x)}, {1.5*sin(\x)});
  \node[] at (-2,0.5) {$P$};
  \draw [orange,line width=10pt, line cap=round, draw opacity=0.7,domain=-88:68] plot ({1.5*cos(\x)}, {1.5*sin(\x)});
  \node[] at (2,-0.5) {$Q$};
  \end{pgfonlayer}

\end{tikzpicture}
\caption{\small The positions of $a_i$ and $a_j$.}\label{figure_positions_of_ai_and_aj}
\end{figure}

Let us prove (\ref{item_thm_bal_balanced}). First we consider the case in which $m$ is odd. For the reverse direction, assume $A$ is not a maximizer of $W$. Since Theorem \ref{theorem_w}(\ref{theorem_w_top_spectra}) fails, there exists an ordered pair $(a_i,a_j)$ of elements of $A$ such that $\spn_A(a_i,a_j)=\floor{\frac{m}{2}}$ and $\d(a_i,a_j)>\floor{\frac{n}{2}}$. It follows that $|\cw{A}{a_i}{a_j}|=\floor{\frac{m}{2}}+1$ and hence $|A\setminus \cw{A}{a_i}{a_j}|=\ceil{\frac{m}{2}}-1$. The set $P=\cw{C_n}{a_i+\floor{\frac{n}{2}}+1}{a_i}$ and its complement $Q=V(C_n)\setminus P$ form an equitable $2$-partition $\P$ with connected parts (see Figure \ref{figure_positions_of_ai_and_aj}). Clearly $a_i\in P$, and since $\d(a_i,a_j)>\floor{\frac{n}{2}}$, it follow that $a_j\in P$. Since $A\setminus \cw{A}{a_i}{a_j}\subseteq P$, $a_i,a_j\in A\cap P$ and $a_i,a_j\notin A\setminus \cw{A}{a_i}{a_j}$, it follows that
\[|A\cap P|\geq |A\setminus \cw{A}{a_i}{a_j}|+2=\ceil{\frac{m}{2}}+1.\] 
Since $a_i,a_j\in P\setminus Q$, it follows that 
\[|A\cap Q|\leq|\cw{A}{a_i}{a_j}|-2=\floor{\frac{m}{2}}-1.\]
Since $m$ is odd, $|A\cap P|$ and $|A\cap Q|$ differ by at least $3$. Thus $\P$ witnesses that $A$ is not balanced.

Still assuming $m$ is odd, let us prove the forward direction of (\ref{item_thm_bal_balanced}). Suppose $A$ is not balanced. Then there is an equitable $2$-partition $\P=\{P,Q\}$ of $V(C_n)$ with connected blocks $P$ and $Q$, where $|P|=\floor{\frac{n}{2}}\leq|Q|=\ceil{\frac{n}{2}}$ such that the cardinalities of $A$ intersected with the two blocks of $\P$ differ by at least $2$, and since $m$ is odd these cardinalities differ by at least $3$. Let $A\cap P=\cw{A}{a_{i_P}}{a_{j_P}}$ and $A\cap Q=\cw{A}{a_{i_Q}}{a_{j_Q}}$ where $i_P,j_P,i_Q,j_Q\in\{0,\ldots,m-1\}$. Then either $|A\cap Q|\leq\floor{\frac{m}{2}}-1$ or $|A\cap Q|\geq\ceil{\frac{m}{2}}+1$. Suppose $|A\cap Q|\leq\floor{\frac{m}{2}}-1$. Then $|\cw{A}{a_{i_Q}}{a_{j_Q+2}}|\leq \floor{\frac{m}{2}}+1$ and hence $\spn_A(a_{i_Q},a_{j_Q+2})\leq\floor{\frac{m}{2}}$. But $\d(a_{i_Q},a_{j_Q+2})\geq |Q|+1=\ceil{\frac{n}{2}}+1$, which implies $A$ is not a maximizer of $W$ since Theorem \ref{theorem_w}(\ref{theorem_w_all_spec}) is false. Now suppose $|A\cap Q|\geq\ceil{\frac{m}{2}}+1$. Then $|\cw{A}{a_{i_P}}{a_{j_P+2}}|=|A\cap P|+2\leq\floor{\frac{m}{2}}+1$, which implies $\spn_A(a_{i_P},a_{j_P+2})\leq\floor{\frac{m}{2}}$. But $\d(a_{i_P},a_{j_P+2})\geq\floor{\frac{n}{2}}+1$, and since $m$ is odd, Theorem \ref{theorem_w}(\ref{theorem_w_all_spec}) is false and hence $A$ is not a maximizer of $W$. This completes the proof of (\ref{item_thm_bal_balanced}) in case $m$ is odd.

%Now let us prove (\ref{item_thm_bal_balanced}) assuming that $n$ is even. As we already handled the case in which $m$ is odd above, we can assume without loss of generality that $m$ is even. For the reverse direction, suppose $A$ is not a maximizer of $W$. Then, since Theorem \ref{theorem_w}(\ref{theorem_w_top_spectra}) is false, there is a pair $(a_i,a_j)\in A\times A$

To finish the proof of (\ref{item_thm_bal_balanced}), it remains to consider the case in which $n$ is even and $m$ is even. For the reverse direction, suppose $A$ is not a maximizer of $W$. Then, since Theorem \ref{theorem_w}(\ref{theorem_w_top_spectra}) is false, $\supp(\spec_{\frac{m}{2}}(A))\not\subseteq\left\{\floor{\frac{n}{2}},\ceil{\frac{n}{2}}\right\}=\left\{\frac{n}{2}\right\}$. It follows that there exists an ordered pair $(a_i,a_j)\in A\times A$ such that $\spn_A(a_i,a_j)=\frac{m}{2}$ and $d^*(a_i,a_j)>\frac{n}{2}$. Notice that $|\cw{A}{a_i}{a_j}|=\frac{m}{2}+1$ and $|A\setminus \cw{A}{a_i}{a_j}|=\frac{m}{2}-1$. Then the set $P=\cw{C_n}{a_i+\frac{n}{2}+1}{a_i}$ and its complement $Q$ form an equitable $2$-partition of $V(C_n)$ with connected parts. Furthermore, since $A\setminus \cw{A}{a_i}{a_j}\subseteq P$ and $a_i,a_j\in P$, it follows that $|P|\geq\frac{m}{2}+1$ and $|Q|\leq\frac{m}{2}-1$. Hence $A$ is not balanced.

Now we prove the forward direction of (\ref{item_thm_bal_balanced}) assuming $n$ and $m$ are even. Suppose $A$ is not balanced. This means that there is an equitable $2$-partition $\P=\{P,Q\}$ of $V(C_n)$ with connected blocks such that $|P|=|Q|=\frac{n}{2}$, and the quantities $|A\cap P|$ and $|A\cap P|$ differ by at least $2$. Without loss of generality, say $|A\cap P|<|A\cap Q|$. Then $|A\cap P|\leq \frac{m}{2}-1$ and $|A\cap Q|\geq\frac{m}{2}+1$. Since the blocks of $\P$ are connected, it follows that $A\cap P=\cw{A}{a_{i_P}}{a_{j_P}}$ and $A\cap Q=\cw{A}{a_{i_Q}}{a_{j_Q}}$ for some $i_P,j_P,i_Q,j_Q\in\{0,\ldots,m-1\}$. Since 
\[|\cw{A}{a_{j_Q}}{a_{i_Q}}|=|A\cap P|+2\leq\frac{m}{2}+1\]
we see that $\spn_A(a_{j_Q},a_{i_Q})\leq\frac{m}{2}$. But $\d(a_{j_Q},a_{i_Q})>\frac{n}{2}$, so $A$ is not a maximizer of $W$ because Theorem \ref{theorem_w}(\ref{theorem_w_top_spectra}) is false. This completes the proof of (\ref{item_thm_bal_balanced}) in the case where $m$ is odd.

Now, to prove (\ref{item_thm_bal_wbnb}), suppose $n$ is odd and $m$ is even. For the reverse direction of the equivalence in (\ref{item_thm_bal_wbnb}), suppose $A$ is not a maximizer of $W$. Then, since Theorem \ref{theorem_w}(\ref{theorem_w_top_spectra}) is false, there is a pair $(a_i,a_j)\in A\times A$ such that $\spn_A(a_i,a_j)=\frac{m}{2}$ and $\d(a_i,a_j)>\ceil{\frac{n}{2}}$. Then $|a_{[i,j]}|=\frac{m}{2}+1$ and $|A\setminus a_{[i,j]}|=\frac{m}{2}-1$. It follows that the set $P=\cw{C_n}{a_i+\ceil{\frac{n}{2}}+1}{a_i}$ and its complement $Q$ form an equitable $2$-partition of $V(C_n)$ with connected parts, where $|P|=\floor{\frac{n}{2}}<|Q|=\ceil{\frac{n}{2}}$. We have 
\[|A\cap P|\geq|A\setminus \cw{A}{a_i}{a_j}|+2=\frac{m}{2}+1\]
and
\[|A\cap Q|\leq|\cw{A}{a_i}{a_j}|-2=\frac{m}{2}-1.\]
So $|A\cap P|$ and $|A\cap Q|$ differ by at least $2$. If these cardinalities differ by $3$ or more, then clearly $A$ is not weakly balanced. If $|A\cap P|$ and $|A\cap Q|$ differ by exactly $2$, then $A$ is not weakly balanced because the smaller block $P$ of the partition contains more elements of $A$ than the larger block.

Still assuming $n$ is odd and $m$ is even, let us prove the forward direction of the equivalence in (\ref{item_thm_bal_wbnb}). Suppose $A$ is not weakly balanced. Thus, we may let $\mathcal{P}=\{P,Q\}$ be an equitable $2$-partition of $V(C_n)$ with connected blocks such that $|P|=\floor{\frac{n}{2}}<|Q|=\ceil{\frac{n}{2}}$ where the cardinalities $|A\cap P|$ and $|A\cap Q|$ differ by at least $3$ or these cardinalities differ by exactly $2$ and the smaller block of the partition contains more members of $A$, that is, $|A\cap P|=|A\cap Q|+2$. Let $A\cap P=\cw{A}{a_{i_P}}{a_{j_P}}$ and $A\cap Q=\cw{A}{a_{i_Q}}{a_{j_Q}}$ where $i_P,j_P,i_Q,j_Q\in\{0,\ldots,m-1\}$. Suppose $|A\cap Q|$ and $|A\cap P|$ differ by at least $3$. Since $m$ is even it follows that $|A\cap Q|$ and $|A\cap P|$ differ by at least $4$.  Then either $|A\cap Q|\leq \frac{m}{2}-2$ or $|A\cap Q|\geq\frac{m}{2}+2$. Suppose $|A\cap Q|\leq \frac{m}{2}-2$, then $\spn_A(a_{i_Q},a_{j_Q})\leq\frac{m}{2}-3$ and hence $\spn_A(a_{i_Q},a_{j_Q+3})\leq\frac{m}{2}$. But $\d(a_{i_Q},a_{j_Q+3})\geq|Q|-1+3=\ceil{\frac{n}{2}}+2$, which implies Theorem \ref{theorem_w}(\ref{theorem_w_all_spec}) is false and hence $A$ is not a maximizer of $W$. Now suppose $|A\cap Q|\geq\frac{m}{2}+2$. Then $|\cw{A}{a_{i_P}}{a_{j_P+3}}|=|A\cap P|+3\leq\frac{m}{2}+1$, which implies that $\spn_A(a_{i_P},a_{j_P+3})\leq\frac{m}{2}$, but $\d(a_{i_P},a_{j_P+3})\geq |P|-1+3=\floor{\frac{n}{2}}+2=\ceil{\frac{n}{2}}+1$. Thus Theorem \ref{theorem_w}(\ref{theorem_w_all_spec}) is false and hence $A$ is not a maximizer of $W$. Now suppose $|A\cap P|=|A\cap Q|+2$. Then $|A\cap P|=\frac{m}{2}+1$ and $|A\cap Q|=\frac{m}{2}-1$. Then $\spn_A(a_{i_P},a_{j_P})=\frac{m}{2}$, but $\d(a_{i_P},a_{j_P})\leq |P|-1=\floor{\frac{n}{2}}-1$, which implies that Theorem \ref{theorem_w}(\ref{theorem_w_top_spectra}) is false and hence $A$ is not a maximizer of $W$.

Finally, to prove the second statement in (\ref{item_thm_bal_wbnb}), suppose $A$ is any maximizer of the Wiener index on $C_n$ where $n$ is odd and $|A|=m$ is even. Let $m=2\ell$ where $\ell\in\Z$. In order to prove $A$ is not balanced we must find an equitable $2$-partition $\P=\{P,Q\}$ of $V(C_n)$ with connected blocks such that $|A\cap P|$ and $|A\cap Q|$ are not within one of each other. By Theorem \ref{theorem_w}(\ref{theorem_w_top_spectra}), we have
\[\supp(\sigma^*_\ell(A))\subseteq\left\{\floor{\frac{n}{2}},\ceil{\frac{n}{2}}\right\}.\]
By Corollary \ref{corollary_sum_of_spec_of_J_rep} we have $\sum\spec_\ell(A)=\ell n$, and since $\max(\spec_\ell(A))-\min(\spec_\ell(A))\leq 1$, it follows by Lemma \ref{lemma_consecutive} that there is a pair $(a_i,a_j)\in A\times A$ such that $\spn_A(a_i,a_j)=\ell$ and $d^*(a_i,a_j)=\floor{\frac{n}{2}}$. Let us take $Q=[a_i,a_j]$ and $P=V(C_n)\setminus [a_i,a_j]$. We have
\[|Q|=d^*(a_i,a_j)+1=\floor{\frac{n}{2}}+1\]
and 
\[|P|=n-|Q|=\ceil{\frac{n}{2}}-1,\]
which implies that $\mathcal{P}=\{P,Q\}$ is an equitable $2$-partition of $V(C_n)$ with connected blocks. Furthermore,
\[|A\cap Q|=|\cw{A}{a_i}{a_j}|=\spn_A(a_i,a_j)+1=\ell+1\]
and
\[|A\cap P|=|A\setminus \cw{A}{a_i}{a_j}|=2\ell-(\ell+1)=\ell-1.\]
Thus $|A\cap P|$ and $|A\cap Q|$ differ by $2$ and hence $A$ is not balanced.
\end{proof}

\section{Future Directions}\label{section_future}

Throughout the preceding sections, we have hinted toward directions for further study, and sticking points for extending our results.  In this short section, we collect a few ideas for future research in this area.

Several results in the current article provide characterizations of minimizers or maximizers of various real-valued functions of the form $E_g$, defined on the powerset of the vertex set of a finite simple connected graph, which allow us to list all minimizers or maximizers of $E_g$ on graphs in the class relatively quickly without resorting to the brute force algorithm. In what other cases are results like this possible?  When can we characterize the minimizers/maximizers of $E_g$ on a particular class of graphs, or even give an efficient algorithm for finding such sets?

\begin{question}\label{question_find_mins}
Given a class of graphs $\mathcal{G}$ and a function $g:[1,\infty)\to\R$, can one find a formula or an algorithm, faster than the brute force algorithm, for generating all minimizers or maximizers of $E_g$ among all graphs in $\mathcal{G}$?
\end{question}

Let us state a few specific instances of Question \ref{question_find_mins} that seem approachable, but which remain open.

\begin{question}\label{question_specific}
Suppose $g:[1,\infty)\to\R$ is a strictly decreasing strictly convex function (or just take $g(r)=\frac{1}{r}$). Which sets of vertices are minimizers $E_g$ on the cylinders $P_m\, \Box\, C_n$? On the toroidal graphs $C_{n_1}\, \Box\, \cdots \, \Box\, C_{n_k}$? Which sets of vertices are maximizers of the Wiener index $W$ on these graphs?
\end{question}

If $G$ is either a path $P_n$ or a cycle $C_n$ on $n$ vertices we showed that, on $G$, every local maximizer of $W$ is a maximizer of $W$ (see Proposition \ref{proposition_wiener_path} and Theorem \ref{theorem_w}). Thus, as noted at the end of Section \ref{section_intro} and in Corollary \ref{corollary_ascent_on_paths} and Corollary \ref{corollary_ascent_on_cycles}, on paths and cycles the algorithm $\ascent$ produces a maximizer of $W$. The corresponding problems involving paths, cycles and $E_g$ for functions other than $g(r)=r$ remain open.

\begin{question}\label{question_local}
Suppose $g:\left\{1,\ldots,\floor{\frac{n}{2}}\right\}\to\R$ is a strictly decreasing convex function (or for concreteness take $g(r)=\frac{1}{r}$) and let $G$ be either a path $P_n$ or a cycle $C_n$ on $n$ vertices. Is every local minimizer of $E_g$ on $G$ also a global minimizer of $E_g$? If $A$ is a set of vertices of $G$, does it follow that $\descent(G,E_g,A)$ is a minimizer of $E_g$?
\end{question}

The authors suspect that the notion of majorization, specifically Corollary \ref{corollary_maj_consistent_with_energy}, will be a useful tool for any future work on Question \ref{question_specific} and Question \ref{question_local}.

%\bibliographystyle{plain}
%\bibliography{newbib}

\end{document}